\documentclass{article}
\usepackage[utf8]{inputenc}
\usepackage[a4paper,margin=23mm,top=30mm]{geometry}
\usepackage{amsfonts, amsmath, amssymb, amsgen, amsthm, amscd,bm,caption,color,comment,enumerate,enumitem,graphicx,hyperref,lmodern,mathrsfs,mathtools,slashed,subcaption,setspace}
\usepackage{arydshln,blkarray}
\usepackage{algorithm}
\usepackage{algpseudocode}
\usepackage{pgfplots}
\pgfplotsset{compat=1.18}
\usepackage{stmaryrd} 
\usepackage[toc,page]{appendix}

\allowdisplaybreaks 
\usepackage{tikz}
\usepackage{ifthen}
\usetikzlibrary{decorations.pathreplacing}

\theoremstyle{plain}
\newtheorem{theorem}{Theorem}
\newtheorem{lemma}{Lemma}
\newtheorem{proposition}{Proposition}

\theoremstyle{definition}
\newtheorem{definition}{Definition}

\newtheorem{assumption}{Assumption}
\newtheorem{example}{Example}

\theoremstyle{remark}
\newtheorem{remark}{Remark}

\newcommand{\ch}[4]{\mathscr{H}_{{#1},\bm{#2}}\left(#4\right)}
\newcommand{\ct}[4]{\mathscr{T}_{{#1},\bm{#2}}\left(#4\right)}

\title{Construction of Lyapunov density for nonautonomous dynamical systems on hypertorus}
\author{
Swapnil Tripathi\footnote{Corresponding author: swap\textunderscore trip@outlook.com}
\ and \"{O}zkan Karabacak\\
Department of Computer Engineering\\
Kadir Has University, Istanbul, Türkiye
}

\begin{document}
\maketitle
\begin{abstract}
We present a semidefinite programming framework for constructing time-varying Lyapunov densities for nonautonomous dynamical systems on a hypertorus. The formulation leverages Gram matrix representations of hybrid (real–trigonometric) polynomials. In addition, we introduce a novel block decomposition of these Gram representations to confine the blow-up of the resulting density to a prescribed set. The results are then applied to establish the almost global synchronization of a time-varying Kuramoto model and the robust almost-global stability of a parameter-varying nonautonomous system. These examples demonstrate the applicability of the proposed method and validate the theoretical results. All computational results are obtained using an open-source MATLAB implementation, as referenced in the text, thereby facilitating reproducibility of the reported examples.
\end{abstract}

\section{Introduction}

Lyapunov density (or dual Lyapunov) methods have emerged as a powerful tool for certifying almost-global stability of dynamical systems. In Rantzer’s seminal result~\cite{rantzer2001dual}, the existence of a density function, which is a positive function (away from the origin) whose divergence along the autonomous flow is positive almost everywhere, is sufficient for almost all trajectories to converge to the origin. The Lyapunov density, which can be interpreted in terms of the continuity equation from fluid mechanics~\cite{meinsma2006rantzer} even on general manifolds~\cite{angeli2003some}, is a perspective dual to that of conventional global Lyapunov functions, which may not exist for systems with complex geometry or with oscillatory modes. Motivated by this dual viewpoint, subsequent works have focused on extending Rantzer’s framework beyond autonomous systems. Initially, Monzón~\cite{monzon2006almost} obtained a sufficient condition for almost-global stability of nonautonomous systems having a locally stable equilibrium point at the origin. Subsequently, Masubuchi and Kikuchi~\cite{masubuchi2017lyapunov,masubuchi2021lyapunov} obtained a sufficient condition without assuming local stability and forward-completeness of the system. In a recent article, Masubuchi~\cite{masubuchi2023extended} introduced additional freedom in the inequalities of Lyapunov densities.

Despite these theoretical advances, constructing Lyapunov densities remains a challenging task in practice. In the autonomous setting, however, structured computational approaches—particularly those based on polynomial~\cite{prajna2004nonlinear}, trigonometric~\cite{tripathi2026certification}, and (real-trigonometric) hybrid~\cite{tripathi_hybrid} representations—have enabled systematic and tractable constructions in several nontrivial classes of systems, even though a general solution remains out of reach. In these results, the density is constructed via sum-of-squares (SOS) techniques, yielding tractable semidefinite programs (SDPs). However, extending such computational frameworks to nonautonomous systems introduces additional challenges due to time dependence. The sufficient condition in~\cite{masubuchi2023extended} requires integrability of an expression, which is in terms of the vector field and the density, over the time variable. As noted in~\cite{masubuchi2023extended}, ``this integrability in time over $\mathbb{R}_+$ would make it difficult to find a Lyapunov density for nonlinear time-varying systems that is not periodic".

In this paper, we obtain SDPs to construct a time-varying Lyapunov density for nonautonomous systems on the hypertorus that are polynomial, periodic, or mixed (polynomial-periodic) in time. We obtain these SDPs in terms of Gram matrix representations~\cite{dumitrescu2007positive,dumitrescu2010positive}, which have been earlier used to obtain densities for autonomous systems on the hypertorus~\cite{tripathi2026certification} and hypercylinders~\cite{tripathi_hybrid}. To utilize Gram representations, the nonautonomous systems in this paper are assumed to be mixed (polynomial-trigonometric) in the time variable $t$. The integrability issue mentioned in~\cite{masubuchi2023extended} is addressed for this class of time-varying systems by imposing that the highest polynomial power of $t$ occurring in the density must be larger than the highest power of $t$ occurring in any component of the vector field. This is enforced in the SDP by constraining the size of the Gram matrix corresponding to the Lyapunov density, which is treated as a decision variable.

The properties of Gram representations are directly tied to the function's zeros and can be used to confine the zero set of its corresponding function. This, in turn, can be used to confine the blow-up of the obtained density to a prescribed set, namely to the equilibrium point, as required in~\cite{masubuchi2023extended}. To aid this confinement of the zero set via SDP, a novel block decomposition of Gram representations is proposed, and an iterated block condition is obtained (Proposition~\ref{prop:PD_condition_gram}). In particular, we prove that a function that is polynomial (trigonometric, respectively) in $t$ vanishes at the origin if the sum of all blocks on any given antidiagonal (diagonal, respectively) is a zero-sum matrix, that is, the sum of all entries of this matrix is zero (Proposition~\ref{prop:PD_condition_gram}, Remark~\ref{rem:zero_sum_hybrid} and Figure~\ref{fig:block_condition_visual}). Imposing an additional constraint on the Gram representation ensures that its corresponding function vanishes \textit{exactly at the origin}. As a consequence, we obtain SDPs (Theorems~\ref{thm:ags} and~\ref{thm:ags_trig}) to construct time-varying density satisfying all sufficient conditions in~\cite{masubuchi2023extended}. These SDPs can be solved using our program, \texttt{NautLDT} (Nonautonomous Lyapunov Density on Torus), a MATLAB-based solver developed using CVX~\cite{cvx} and publicly available on GitHub~\cite{Tripathi_Software_Nonautonomous_Lyapunov_2026}. 

The remainder of the paper is organized as follows. In Section~\ref{sec:notations}, we introduce the notations that will be used throughout the paper. In Section~\ref{sec:background}, we present the sufficient conditions for almost-global stability of time-varying systems obtained in~\cite{masubuchi2023extended}. In Section~\ref{sec:SDP_formul_hybrid}, we present some results on Gram representations in Section~\ref{sec:known_results_hybrid}, define the block decomposition of a Gram representation and derive a sufficient condition for positive definiteness of its corresponding function in Section~\ref{sec:block_decomp_hybrid}; and subsequently obtain the SDP formulation of~\cite{masubuchi2023extended} for systems which are polynomial, trigonometric or mixed (polynomial-trigonometric) in $t$ in Section~\ref{sec:SDP_formulation_thm_hybrid}. In Section~\ref{sec:examples}, we apply the theory to obtain a Lyapunov density for a system that is polynomial in $t$ in Example~\ref{examp:1}, and to determine almost-synchronization in a Kuramoto model with interconnection-dependent time-periodic coupling strengths in Example~\ref{examp:Kuramoto}. We conclude the paper by applying the theory to robust stabilization of vector fields that are affine in the uncertain parameter (Example~\ref{examp:2}) in Section~\ref{sec:robust}.

\subsection{Notations}\label{sec:notations}

The following notations are used throughout the paper. The $n-$dimensional hypertorus is denoted by $\mathbb{T}^n=[0,2\pi)^n$. The set of continuously differentiable functions from space $X$ to $Y$ is denoted by $\mathcal{C}^1(X,Y)$. For any function $v(t,\bm{\theta})\in\mathcal{C}^1(\mathbb{R}\times\mathbb{T}^n,\mathbb{R})$, its partial derivatives are denoted by $v_t(t,\bm{\theta})$, $v_{\theta_1}(t,\bm{\theta}),\ldots,v_{\theta_n}(t,\bm{\theta})$, and its gradient in space variables is expressed as $\operatorname{grad}_{\bm{\theta}}\, v(t,{\bm{\theta}})\vcentcolon=\left[\frac{\partial v}{\partial \theta_1}(t,\bm{\theta}),\ldots,\frac{\partial v}{\partial \theta_n}(t,\bm{\theta})\right]$. For a nonautonomous vector field $\dot{\bm{\theta}}=\bm{f}(t,{\bm{\theta}})$ on $\mathbb{T}^n$, the $l^{th}$ component of the vector field is denoted by $f^{(l)}(t,\bm\theta)$, and the divergence of the vector field in space variables is expressed as $\operatorname{div}_{\bm{\theta}}\, \bm{f}(t,{\bm{\theta}})\vcentcolon=\sum_{l=1}^n f^{(l)}_{\theta_l}(t,\bm{\theta})$. The solution curve of the nonautonomous system with initial condition $\bm{\theta_0}$ at time $t_0$ is denoted by $\bm\Phi(t;t_0,\bm{\theta_0})$. The vectors of zeroes and ones are denoted by $\bm{0}$ and $\bm{1}$, respectively, and their dimensions are understood from the context. The Euclidean norm on $\mathbb{T}^n$ is denoted by $\|\cdot\|$ and ball of radius $a$ in $\mathbb{T}^n$ centered at the origin is denoted by $\mathbb{B}^n(a;\bm{0})$. Submatrices of a matrix are defined using an index-set notation, meaning, for $A\in \mathbb{R}^{m\times n}$, $I\subseteq\left\{1, \ldots, m\right\}$, and $J\subseteq\left\{1, \ldots, n\right\}$, the submatrix formed by rows in $I$ and columns in $J$ is denoted by $A_{I,J}$. The conjugate transpose of $A$ is denoted by $A^\dagger$. The Kronecker product of square matrices $A$ and $B$ is denoted by $A\otimes B$, and their Kronecker sum is denoted by $A\oplus B=A\otimes I_{\textup{size}(B)} + I_{\textup{size}(A)}\otimes B$. For vectors $\bm{k},\bm{j}\in\mathbb{Z}^n$, the inequality $\bm{k}\ge \bm{j}$ is understood elementwise.

\section{Background and problem statement}\label{sec:background}

Consider the nonautonomous system 
\begin{equation}\label{eq:system_timevar}
    \dot{{\bm{\theta}}}=\bm{f}(t,{\bm{\theta}})
\end{equation}
where the origin ${\bm{\theta}}=\bm{0}$ is an \textit{equilibrium} of the system, meaning $\bm{f}(t,{\bm{0}})={\bm{0}}$ for all $t\in\mathbb{R}$; and $\bm{f}$ is \textit{locally Lipschitz continuous} in $\bm{\theta}$ at $(t,\bm{0})$ for all $t\in\mathbb{R}$, that is, for each $t\in\mathbb{R}$, there exists $a>0$ and $L>0$ such that $\left\| \bm{f}(\tau,\bm{\theta^{(1)}})-\bm{f}(\tau,\bm{\theta^{(2)}})\right\|<L\left\|\bm{\theta^{(1)}}-\bm{\theta^{(2)}}\right\|$ holds for all $\tau\in(t-a,t+a)$ and for all $\bm{\theta^{(1)}},\bm{\theta^{(2)}}\in \mathbb{B}^n(a;\bm{0})$. The solutions of nonautonomous systems can be expressed using a continuous mapping $(t,t_0,\bm{\theta_0})\mapsto \bm\Phi(t;t_0,\bm{\theta_0})$ for all $(t_0,\bm{\theta_0})\in \mathbb{R}\times \mathbb{T}^n$ and $t$ in some maximal interval $(t_-,t_+)$. The following theorem gives a sufficient condition to check that all solutions are forward-complete, meaning $t_+=\infty$, and almost all of them converge to the origin.

\begin{lemma}[Adapted from~\cite{masubuchi2023extended}]\label{lem:ags_general}
    Assume that $\bm{f}\in\mathcal{C}^1(\mathbb{R}\times \mathbb{T}^n,\mathbb{T}^n)$ in~\eqref{eq:system_timevar} satisfies $\bm{f}(t,{\bm{0}})={\bm{0}}$ for all $t\in\mathbb{R}$, and is locally Lipschitz continuous in $\bm{\theta}$ at $(t,\bm{0})$ for all $t\in\mathbb{R}$. Suppose there exists a nonnegative function $\rho\in\mathcal{C}^1(\mathbb{R}\times\left(\mathbb{T}^n\setminus \{\bm{0}\}\right),\mathbb{R})$ such that \begin{equation}\label{eq:integral_hybrid}
        \widehat{I}_r=\int_0^\infty \int_{\left\|\bm{\theta}\right\|\ge r}\frac{1+\left\|\bm{f}(t,\bm{\theta})\right\|}{1+\|\bm{\theta}\|}\rho(t,\bm{\theta})\ \mathrm{d}\bm{\theta}\, \mathrm{d}t
    \end{equation}
    is finite for all $r>0$, and  \begin{equation}\label{eq:monzon}
        \frac{\partial \rho}{\partial t}(t,{\bm{\theta}})+\operatorname{div}_{\bm{\theta}} (\rho\bm{f})(t,{\bm{\theta}})>0
    \end{equation}
    for almost all $(t,{\bm{\theta}})\in \mathbb{R}_+\times \mathbb{T}^n$. Then, for almost all $(t_0,\bm{\theta_0})\in\mathbb{R}_+\times \mathbb{T}^n$, solution $\bm\Phi(t;t_0,\bm{\theta_0})$ is defined for all $t\ge t_0$ and $\lim_{t\to\infty}\bm\Phi(t;t_0,\bm{\theta_0})=\bm{0}$.
\end{lemma}
The following theorem gives a sufficient condition to check that all solutions are forward-complete and almost all of them converge to the origin for periodic systems.
\begin{lemma}[Adapted from~\cite{masubuchi2023extended}]\label{lem:ags_general_trigo}
    Given $T>0$, if $\bm{f}\in\mathcal{C}^1(\mathbb{R}\times \mathbb{T}^n,\mathbb{T}^n)$ in~\eqref{eq:system_timevar} satisfies $\bm{f}(t,{\bm{0}})={\bm{0}}$ and $\bm{f}(t+T,\bm\theta)=\bm{f}(t,\bm{\theta})$ for all $t\in\mathbb{R}$, and is locally Lipschitz continuous in $\bm{\theta}$ at $(t,\bm{0})$ for all $t\in\mathbb{R}$. Suppose there exists a nonnegative function $\rho\in\mathcal{C}^1(\mathbb{R}\times\left(\mathbb{T}^n\setminus \{\bm{0}\}\right),\mathbb{R})$ satisfying $\rho(t+T,\bm\theta)=\rho(t,\bm{\theta})$ for all $t\in\mathbb{R}$ such that \begin{equation}\label{eq:integral_trig}
        {\widehat{I}_r}^{rot}=\int_{\left\|\bm{\theta}\right\|\ge r}\frac{1+\left\|\bm{f}(t,\bm{\theta})\right\|}{1+\|\bm{\theta}\|}\  \rho(t,\bm{\theta})\  \mathrm{d}\bm{\theta}
    \end{equation} 
    is finite for all $r,t>0$, and~\eqref{eq:monzon} is satisfied for almost all $(t,{\bm{\theta}})\in [0,T]\times \mathbb{T}^n$. Then, for almost all $(t_0,\bm{\theta_0})\in\mathbb{R}_+\times \mathbb{T}^n$, solution $\bm\Phi(t;t_0,\bm{\theta_0})$ is defined for all $t\ge t_0$ and $\lim_{t\to\infty}\bm\Phi(t;t_0,\bm{\theta_0})=\bm{0}$. 
\end{lemma}

Lemmas~\ref{lem:ags_general} and~\ref{lem:ags_general_trigo} were originally proved for nonautonomous systems on $\mathbb{R}^n$, but they extend to smooth manifolds such as the hypertorus $\mathbb{T}^n$ and hypercylinder $\mathbb{R}^m\times\mathbb{T}^{n-m}$. Indeed, these manifolds admit a natural smooth volume form (e.g., induced by a Riemannian metric), with respect to which the divergence operator is defined intrinsically via the Lie derivative. Consequently, the continuity equation~\eqref{eq:monzon} is well-posed in a coordinate-free manner, and the proof carries over verbatim provided the integrability assumption~\eqref{eq:integral_hybrid}, or~\eqref{eq:integral_trig}, respectively, is satisfied. 

The nonautonomous system~\eqref{eq:system_timevar} evolves on a manifold with both Euclidean and periodic components when time dependence is not periodic, hence motivating the use of hybrid (real–trigonometric) polynomials~\cite{dumitrescu2010positive}. In particular, the Euclidean component captures the time dependence, while the trigonometric part encodes the periodic geometry of $\mathbb{T}^{n}$. Such representations therefore provide a convenient framework for constructing density functions and for verifying the continuity equation~\eqref{eq:monzon}.

\begin{definition}[Functions mixed polynomial-trigonometric in $t$ and trigonometric in $\bm\theta$]\label{def:hybrid}
A function $v(t,\bm{\theta})$ is said to be mixed (polynomial-trigonometric) in $t$ and trigonometric in $\bm\theta$ if it is defined on $\mathbb{R}\times\mathbb{T}^n$ as \begin{equation}\label{eq:polynomial_mix}
		\begin{aligned}
		    v(t,\bm{\theta})&=\sum_{m_1=0}^ {deg_{v_{t}}^{pol}}\sum_{m_2=-deg_{v_{t}}^{rot}}^ {deg_{v_{t}}^{rot}}\,\sum_{\bm{k}=-\bm{deg_{v_\theta}}}^{\bm{deg_{v_\theta}}}v_{m_1,m_2,\bm{k}}\,t^{m_1}\,{\rm{e}}^{-i\,(m_2,\bm{k})\cdot(t,\bm{\theta})}\\
            &=\sum_{m_1=0}^ {deg_{v_{t}}^{pol}}\sum_{m_2=-deg_{v_{t}}^{rot}}^ {deg_{v_{t}}^{rot}}\,\sum_{k_1=-{deg_{v_\theta}(1)}}^{{deg_{v_\theta}(1)}}\ldots\sum_{k_n=-{deg_{v_\theta}(n)}}^{{deg_{v_\theta}(n)}} v_{m_1,m_2,k_1,\ldots, k_n}\,t^{m_1}\,{\rm{e}}^{-i\,(m_2 t+k_1\theta_1+\ldots k_n\theta_n)},
		\end{aligned}
\end{equation}
where the coefficients satisfy $v_{m_1,-m_2,-\bm{k}}=\overline{v_{m_1,m_2,\bm{k}}}$ so as to make the function real-valued. Here, the vector $\bm{deg_{v}}\vcentcolon=(deg_{v_{t}}^{pol},deg_{v_{t}}^{rot},\bm{deg_{v_{\theta}}})\ge\bm{0}$ is taken as minimal and is called the \textit{degree} of $v$. If $deg_{v_{t}}^{rot}=0$, the function $v(t,\bm{\theta})$ is \textit{polynomial in $t$ and trigonometric in $\bm\theta$}, and takes the form \begin{equation*}
		\begin{aligned}
		    v(t,\bm{\theta})&=\sum_{m_1=0}^ {deg_{v_{t}}^{pol}}\,\sum_{\bm{k}=-\bm{deg_{v_\theta}}}^{\bm{deg_{v_\theta}}}v_{m_1,0,\bm{k}}\,t^{m_1}\,{\rm{e}}^{-i\,\bm{k}\cdot\bm{\theta}}.
		\end{aligned}
\end{equation*}
Similarly, if $deg_{v_{t}}^{pol}=0$, the function $v(t,\bm{\theta})$ is \textit{trigonometric in $t$ and $\bm\theta$}, and takes the form \begin{equation*}
		\begin{aligned}
		    v(t,\bm{\theta})&=\sum_{m_2=-deg_{v_{t}}^{rot}}^ {deg_{v_{t}}^{rot}}\,\sum_{\bm{k}=-\bm{deg_{v_\theta}}}^{\bm{deg_{v_\theta}}}v_{0,m_2,\bm{k}}\,{\rm{e}}^{-i\,(m_2,\bm{k})\cdot(t,\bm{\theta})}.
		\end{aligned}
\end{equation*}

We will say that the system $\dot{\bm\theta}=\bm{f}(t,\bm\theta)$ or the vector field $\bm{f}$ is polynomial, trigonometric, or mixed in $t$ if all its components $f^{(l)}$ share \textit{the same behavior} (are polynomial, trigonometric, or mixed) in $t$.
\end{definition}

\begin{assumption}
    The nonautonomous system~\eqref{eq:system_timevar} is mixed (polynomial-trigonometric) in $t$.
\end{assumption}

\begin{definition}[Positive definite function]\label{def:positivedefinitehybrid}
    A function $v(t,\bm{\theta})$ is positive semidefinite if $v(t,\bm{\theta})\ge 0$. We say that $v(t,\bm{\theta})$ is positive definite if it is positive semidefinite and $v(t,\bm{\theta})= 0$ if and only if $\bm{\theta}=\bm{0}$.
\end{definition}

These definitions allow us to obtain the following sufficient condition for almost global stability.

\begin{lemma}\label{lem:ags_polynomial}
	Consider a nonautonomous system~\eqref{eq:system_timevar} which is polynomial, trigonometric, or mixed in $t$; and trigonometric in $\bm\theta$. If there exists a nonzero positive definite function $v(t,\bm{\theta})$ and a nonzero positive semidefinite function $w(t,\bm{\theta})$, both sharing the same behavior as $\bm{f}$ in $t$ and are trigonometric in $\bm\theta$; such that either
		\begin{enumerate}[label=(IC\arabic*)]
            \item\label{item:Rintegral} the integral
            \begin{equation}\label{eq:integral_hybridpoly}
        \widehat{I}_r=\int_0^\infty \int_{\left\|\bm{\theta}\right\|\ge r}\frac{1+\left\|\bm{f}(t,\bm{\theta})\right\|}{(1+\|\bm{\theta}\|)\, v(t,\bm{\theta})}\ \mathrm{d}\bm{\theta}\, \mathrm{d}t
    \end{equation}
		\end{enumerate}
        \noindent is finite for all $r>0$, when $\bm{f}$ is polynomial or mixed in $t$, or
    \begin{enumerate}[label=(IC\arabic*')]
        \item\label{item:Tintegral} the integral
            \begin{equation}\label{eq:integral_trigpoly}
        {\widehat{I}_r}^{rot}=\int_{\left\|\bm{\theta}\right\|\ge r}\frac{1+\left\|\bm{f}(t,\bm{\theta})\right\|}{\left(1+\|\bm{\theta}\|\right) v(t,\bm{\theta})}\  \  \mathrm{d}\bm{\theta}
    \end{equation}
    \end{enumerate}
    \noindent is finite for all $r,t>0$, when $\bm{f}$ is trigonometric in $t$; and
    \begin{equation}\label{eq:LieRelationDual}
			w(t,\bm{\theta})=-v_t(t,\bm{\theta})-\operatorname{grad}_{\bm{\theta}} v(t,\bm{\theta})\cdot \bm{f}(t,\bm{\theta}) + v(t,\bm{\theta})\, \operatorname{div}_{\bm{\theta}} \bm{f}(t,\bm{\theta}),
	\end{equation}
    then almost all solutions of~\eqref{eq:system_timevar} satisfy $\lim_{t\to\infty}\bm\Phi(t;t_0,\bm{\theta_0})=\bm{0}$.
\end{lemma}

\begin{proof}
    The functions that are trigonometric in $\bm\theta$ and polynomial, trigonometric, or a mix of both in $t$ are special cases of hybrid polynomials~\cite{dumitrescu2010positive}. Since the zero set of any hybrid polynomial necessarily has zero Lebesgue measure~\cite[Lemma 3]{tripathi_hybrid}, defining $\rho(t,\bm{\theta})=1/v(t,\bm{\theta})$, it follows that $\rho\in\mathcal{C}^1(\mathbb{R}\times\left(\mathbb{T}^n\setminus \{\bm{0}\}\right),\mathbb{R})$;
    \begin{eqnarray*}
			\frac{\partial \rho}{\partial t}(t,{\bm{\theta}})+\operatorname{div}_{\bm{\theta}} (\rho\bm{f})(t,{\bm{\theta}})=\,\frac{w(t,\bm{\theta})}{v(t,\bm{\theta})^2}>0
		\end{eqnarray*}
        for almost all $(t,{\bm{\theta}})\in \mathbb{R}_+\times \mathbb{T}^n$; and integrability conditions are satisfied. The result follows from Lemma~\ref{lem:ags_general} or Lemma~\ref{lem:ags_general_trigo}, depending on the case.
    \end{proof}

\section{SDP formulation of density via Gram representations}\label{sec:SDP_formul_hybrid}

In this section, we obtain an SDP formulation of density for systems that are mixed (polynomial-trigonometric) in $t$ and trigonometric in $\bm\theta$. The functions~\eqref{eq:polynomial_mix} of degree $\bm{deg_{v}}=(deg_{v_{t}}^{pol},deg_{v_{t}}^{rot},\bm{deg_{v_{\theta}}})$ have a matrix representation with respect to a standard basis of monomials~\cite{dumitrescu2010positive} given as columns of the vector
\begin{eqnarray}\label{eq:Standard_basis_hybrid}
	\nonumber\psi_{\bm{n_v}}(t,\bm{\theta})&=& \left[1\, t\, \dots\, t^{n_{v_t}^{pol}}\right]\otimes \left[1\, {\rm{e}}^{i\,t}\,\dots \, {\rm{e}}^{i\,n_{v_t}^{rot}\,t}\right] \otimes \left(\bigotimes_{j=n}^{1}\left[1\, {\rm{e}}^{i\,\theta_j}\,\dots \, {\rm{e}}^{i\,n_{v_\theta}(j)\,\theta_j}\right]\right)\\
    &=\vcentcolon& \widetilde{\psi}_{n_{v_t}^{pol}}(t)\otimes \psi_{n_{v_t}^{rot}}(t) \otimes \psi_{\bm{n_{v_\theta}}}(\bm{\theta}),
\end{eqnarray}
where $n_{v_{t}}^{pol}=\lceil deg_{v_t}^{pol}/2\rceil$, $n_{v_{t}}^{rot}=deg_{v_t}^{rot}$, $\bm{n_{v_{\theta}}}=\bm{deg_{v_{\theta}}}$; and $\bm{n_{v}}=(n_{v_{t}}^{pol},n_{v_{t}}^{rot},\bm{n_{v_{\theta}}})$ is called the \textit{representation-size vector} of $v$.
The Kronecker product notation $\otimes_{j=n}^{1}$ is to be read as the index $n$ appearing in the leftmost position.

A \textit{Gram matrix representation}~\cite{dumitrescu2010positive} associated with the function $v(t,\bm{\theta})$ in~\eqref{eq:polynomial_mix} is defined as a Hermitian matrix $V$ of size $\left\llbracket \bm{n_v}\right\rrbracket\vcentcolon= \left(n_{v_t}^{pol}+1\right) \left(n_{v_t}^{rot}+1\right) \Big(\prod_{j=1}^n\left(n_{v_\theta}(j)+1\right)\Big)$ satisfying
\begin{equation}\label{eq:gram_matrix}
	v(t,\bm{\theta})=\psi_{\bm{n_v}}(t,\bm{\theta})^\dagger\,V\,\psi_{\bm{n_v}}(t,\bm{\theta}).   
\end{equation}
The coefficients of $v(t,\bm\theta)$ are known to be linear functions of any Gram representation $V$ due to \textit{trace parameterization property}~\cite{dumitrescu2010positive}, which states that \begin{eqnarray}\label{eq:trace_param_hybrid}
     v_{m_1,m_2,\bm{k}}=\ch{m_1,m_2}{k}{n_v}{V}\vcentcolon= \operatorname{Tr}\left(\left(B_{m_1}^{n_{{v}_t}^{pol}+1}\otimes T_{m_2}^{n_{v_t}^{rot}+1}\otimes T_{k_n}^{n_{v_\theta}(n)+1}\otimes\ldots\otimes T_{k_1}^{n_{v_\theta}(1)+1}\right)V\right),
\end{eqnarray}
where $T_k^n$ is the $k^{th}$ elementary \textit{Toeplitz matrix} of size $n$, that is, a $(0,1)-$matrix satisfying $\left(T_k^n\right)_{(i,j)}=1$ if and only if $j-i=k$; and  $B_k^n$ is the $k^{th}$ elementary \textit{Hankel matrix} of size $n$, that is, a $(0,1)-$matrix satisfying $\left(B_k^n\right)_{(i,j)}=1$ if and only if $i+j-2=k$. Thus, $B_0^1=T_0^1=[1]$ and $B_k^1=T_k^1=[0]$ for all $k\ne 0$, and the expression~\eqref{eq:trace_param_hybrid} can be simplified for the case when $v(t,\bm\theta)$ is polynomial or trigonometric in $t$. Also, two Gram representations $A$ and $B$ correspond to the same function if and only if $\ch{m_1,m_2}{k}{n_v}{A}=\ch{m_1,m_2}{k}{n_v}{B}$ for all $m_1,m_2,\bm{k}$. 

\subsection{Positive semi-definiteness and derivative relations in Gram representations}\label{sec:known_results_hybrid}

In this section, we present some results on Gram representations. The first result by Dumitrescu~\cite{dumitrescu2010positive} provides a way to enforce positive semidefiniteness of functions of the form~\eqref{eq:polynomial_mix} in a convex program. 

\begin{lemma}[Consequence of~\protect{\cite[Theorem 2]{dumitrescu2010positive}}]\label{lem:PD_dumitrescu}
The function $v(t,\bm{\theta})$ of the form~\eqref{eq:polynomial_mix} is sum-of-squares if and only if there exists a positive semidefinite matrix $V$ such that~\eqref{eq:gram_matrix} holds.    
\end{lemma}

The following result gives a Gram representation of all partial derivatives of a function of the form~\eqref{eq:polynomial_mix} in terms of the Gram representation of the function itself. It is a generalization of~\cite[Proposition 1]{tripathi_hybrid}, when the state space has no Euclidean component, to functions that are mixed (polynomial-trigonometric) in $t$.

\begin{lemma}[]\label{lem:derivative}
	A Gram representation of $v_{t}(t,\bm\theta)$ is given by $D_{t}^\dagger\, V\,+V D_{t}$. Moreover, a Gram representation of $v_{\theta_l}(t,\bm\theta)$ is given by $D_{\theta_l}^\dagger\, V\,+V D_{\theta_l}$, where \begin{eqnarray*}
		D_{\theta_l}&\vcentcolon=&I_{n_{v_t}^{pol}+1}\otimes I_{n_{v_t}^{rot}+1}\otimes\left(\bigotimes_{j=n}^{l+1} I_{n_{v_\theta}(j)+1}\right)\otimes\, M_{n_{v_\theta}(l)}\,\otimes \left(\bigotimes_{j=l-1}^{1} I_{n_{v_\theta}(j)+1}\right),\\
		D_{t}&\vcentcolon=& \left(\widetilde{M}_{n_{v_t}^{pol}}\oplus{M}_{n_{v_t}^{rot}}\right) \otimes \left(\bigotimes_{j=n}^{1} I_{n_{v_\theta}(j)+1}\right),\text{ where}\\
         \widetilde{M}_k&=&\begin{cases}\begin{bmatrix}
			0 & 0 & 0 & \cdots & 0 \\
			1 & 0 & 0 & \ddots & 0 \\
			0 & 2 & 0 & \ddots & 0 \\
			\vdots & \vdots & \ddots & \ddots & \vdots \\
			0 & 0 & \cdots & k & 0 
		\end{bmatrix}&,\, k\ge 1\\
        [0] &,\, k=0\end{cases},\text{ and }  {M}_{k}=\begin{cases}\begin{bmatrix}
			0 & 0 & 0 &\cdots & 0 \\
			0 & i & 0 & \ddots & 0 \\
            0 & 0 & 2i & \ddots & 0 \\
			\vdots & \vdots & \ddots & \ddots & \vdots \\
			0 & 0 & \cdots & 0 & k i 
		\end{bmatrix}&,\, k\ge 1\\
        [0] &,\, k=0\end{cases}.
\end{eqnarray*}  
\end{lemma}

\begin{proof}
We first prove the result for $v_t(t,\bm{\theta})$. Differentiating~\eqref{eq:Standard_basis_hybrid} with respect to $t$, we get
\begin{eqnarray*}
\frac{\partial \psi_{\bm{n_v}}}{\partial t}&=&\frac{\partial\widetilde{\psi}_{n_{v_t}^{pol}}}{\partial t}(t)\otimes \psi_{n_{v_t}^{rot}}(t) \otimes \psi_{\bm{n_{v_\theta}}}(\bm{\theta})+ \widetilde{\psi}_{n_{v_t}^{pol}}(t)\otimes \frac{\partial{\psi}_{n_{v_t}^{rot}}}{\partial t}(t) \otimes \psi_{\bm{n_{v_\theta}}}(\bm{\theta})\\
&=&\left(\widetilde{M}_{n_{v_t}^{pol}}\,\widetilde{\psi}_{n_{v_t}^{pol}}(t)\right)\otimes \psi_{n_{v_t}^{rot}}(t) \otimes \psi_{\bm{n_{v_\theta}}}(\bm{\theta})+ \widetilde{\psi}_{n_{v_t}^{pol}}(t)\otimes \left({M}_{n_{v_t}^{rot}}\,{\psi}_{n_{v_t}^{rot}}(t)\right) \otimes \psi_{\bm{n_{v_\theta}}}(\bm{\theta})\\
&=& \left(\widetilde{M}_{n_{v_t}^{pol}}\otimes I_{n_{v_t}^{rot}+1} \otimes \left(\bigotimes_{j=n}^{1} I_{n_{v_\theta}(j)+1}\right)+I_{n_{v_t}^{pol}+1}\otimes {M}_{n_{v_t}^{rot}} \otimes \left(\bigotimes_{j=n}^{1} I_{n_{v_\theta}(j)+1}\right)\right)\psi_{\bm{n_v}}\\
&=&  \left(\left(\widetilde{M}_{n_{v_t}^{pol}}\otimes I_{n_{v_t}^{rot}+1}+I_{n_{v_t}^{pol}+1}\otimes {M}_{n_{v_t}^{rot}}\right) \otimes \left(\bigotimes_{j=n}^{1} I_{n_{v_\theta}(j)+1}\right)\right)\psi_{\bm{n_v}}\\
&=&\left(\left(\widetilde{M}_{n_{v_t}^{pol}}\oplus{M}_{n_{v_t}^{rot}}\right) \otimes \left(\bigotimes_{j=n}^{1} I_{n_{v_\theta}(j)+1}\right)\right)\psi_{\bm{n_v}}\\
&=& D_t \psi_{\bm{n_v}}.
\end{eqnarray*}
Thus, $v_t=\frac{\partial \psi_{\bm{n_v}}^\dagger}{\partial t}V\psi_{\bm{n_v}}+\psi_{\bm{n_v}}^\dagger V\frac{\partial \psi_{\bm{n_v}}}{\partial t}=\psi_{\bm{n_v}}^\dagger(D_t^\dagger V+V D_t) \psi_{\bm{n_v}}$. Similarly, the second part follows since $\frac{\partial \psi_{\bm{n_v}}}{\partial\theta_l}=D_{\theta_l}\psi_{\bm{n_v}}$.
\end{proof}

The following result gives a Gram representation of the product of functions of the form~\eqref{eq:polynomial_mix} in terms of Gram representations of the individual functions.

\begin{lemma}[\protect{\cite[Proposition 2]{tripathi_hybrid}}]\label{lem:product}
	A Gram representation of the product of functions $v(t,\bm{\theta})$ and $p(t,\bm{\theta})$ of the form~\eqref{eq:polynomial_mix} is given as $X^\dagger\left(V\otimes P\right) X$, where $X$ is a $\left\llbracket\bm{n_v}\right\rrbracket\left\llbracket\bm{n_p}\right\rrbracket\times \left\llbracket\bm{n_v}+\bm{n_p}\right\rrbracket$ matrix defined as \begin{equation}\label{eq:selection_matrix}
	X_{i,j}=\begin{cases}
		1, &\text{ if }(\psi_{\bm{n_v}} \otimes \psi_{\bm{n_p}})_i=(\psi_{\bm{n_v}+\bm{n_p}})_j\\
		0, &\text{ otherwise}
	\end{cases}.
\end{equation}
\end{lemma}

\begin{proof}
    The idea of the proof originally appears in~\cite[Proposition 2]{tripathi_hybrid}, but since it is brief, we reproduce it here for the reader's convenience. The matrix $X$ satisfies $\psi_{\bm{n_v}} \otimes \psi_{\bm{n_p}} = X\, \psi_{\bm{n_v}+\bm{n_p}}$, hence $v(t,\bm{\theta}) p(t,\bm{\theta})=(\psi_{\bm{n_v}}^\dagger\, V\,  \psi_{\bm{n_v}}) (\psi_{\bm{n_p}}^\dagger\, P\,  \psi_{\bm{n_p}})=(\psi_{\bm{n_v}} \otimes \psi_{\bm{n_p}})^\dagger (V \otimes P) (\psi_{\bm{n_v}} \otimes \psi_{\bm{n_p}})= \psi_{\bm{n_v}+\bm{n_p}}^\dagger\, X^\dagger(V \otimes P)X \, \psi_{\bm{n_v}+\bm{n_p}}$.
\end{proof}

The above lemmas allow us to convert~\eqref{eq:LieRelationDual} into a convex constraint. We present the result below. 

\begin{proposition}\label{prop:density_condition}
    Given a nonautonomous system~\eqref{eq:system_timevar} which is mixed (polynomial-trigonometric) in $t$ such that all its components $f^{(l)}$ can be represented by matrices of size $\left\llbracket\bm{n_f}\right\rrbracket$, the functions $v(t,\bm{\theta})$ and $w(t,\bm{\theta})$, having the same behavior in $t$ as the system, satisfy the equality~\eqref{eq:LieRelationDual} if their Gram matrices $V$ and $W$ are of sizes $\left\llbracket\bm{n_v}\right\rrbracket$ and $\left\llbracket\bm{n_v}+\bm{n_f}\right\rrbracket$, respectively; and satisfy \begin{eqnarray}\label{eq:LieRelatonDual_Gram}
        \nonumber\ch{m_1,m_2}{k}{n_w}{W+X^\dagger\left(\sum_{l=1}^n\left(D_{\theta_l}^\dagger V+VD_{\theta_l}\right)\otimes F^{(l)}-\sum_{l=1}^n V\otimes\left(D_{\theta_l}^\dagger F^{(l)}+F^{(l)}D_{\theta_l}\right)\right)X}\\
        \hspace{60pt} =\ch{m_1,m_2}{k}{n_v}{-D_t^\dagger V-V D_t}
    \end{eqnarray}
    for all $0\le m_1\le n_{v_t^{pol}}+n_{f_t^{pol}}$, $\lvert m_2\rvert \le n_{v_t^{rot}}+n_{f_t^{rot}}$, and $-(\bm{n_{v_{\theta}}}+\bm{n_{f_{\theta}}})\le \bm{k}\le \bm{n_{v_{\theta}}}+\bm{n_{f_{\theta}}}$.\\
    \noindent Note: The above equation is a convexification of the continuity equation~\eqref{eq:monzon} and is an extension of its counterpart in~\cite[Theorem 1]{tripathi_hybrid} to nonautonomous systems when the state space has no Euclidean component. Additionally, since the operators $\ch{m_1,m_2}{k}{}{\cdot}$ depend on the representation-size vector of the input matrix, the Toeplitz/Hankel factors occurring due to~\eqref{eq:trace_param_hybrid} in the left-hand side and right-hand side of~\eqref{eq:LieRelatonDual_Gram} have distinct sizes.
\end{proposition}

\begin{proof}
 Using Lemma~\ref{lem:derivative}, $-v_t(t,\bm{\theta})=-\psi_{\bm{n_v}}^\dagger \left(D_t^\dagger V+V D_t\right) \psi_{\bm{n_v}}$.
 The matrices on the left and right sides of~\eqref{eq:LieRelatonDual_Gram} have different sizes, but they give rise to the same function since all their trace parameterizations coincide. Thus using~\eqref{eq:LieRelatonDual_Gram}, \begin{eqnarray*}
     -v_t(t,\bm{\theta})&=&\psi_{\bm{n_v+n_f}}^\dagger\left(W+X^\dagger\left(\sum_{l=1}^n\left(D_{\theta_l}^\dagger V+VD_{\theta_l}\right)\otimes F^{(l)}-\sum_{l=1}^n V\otimes\left(D_{\theta_l}^\dagger F^{(l)}+F^{(l)}D_{\theta_l}\right)\right)X\right)\psi_{\bm{n_v}+\bm{n_f}}\\
     &=&\psi_{\bm{n_v+n_f}}^\dagger\,W\,\psi_{\bm{n_v+n_f}}+\psi_{\bm{n_v+n_f}}^\dagger\left(X^\dagger\left(\sum_{l=1}^n\left(D_{\theta_l}^\dagger V+VD_{\theta_l}\right)\otimes F^{(l)}\right)X\right)\psi_{\bm{n_v+n_f}}\\
     & & -\psi_{\bm{n_v+n_f}}^\dagger\left(X^\dagger\left(\sum_{l=1}^n V\otimes\left(D_{\theta_l}^\dagger F^{(l)}+F^{(l)}D_{\theta_l}\right)\right)X\right)\psi_{\bm{n_v+n_f}}\\
     &=& w(t,\bm{\theta})+\operatorname{grad}_{\bm{\theta}} v(t,\bm{\theta})\cdot \bm{f}(t,\bm{\theta}) - v(t,\bm{\theta})\, \operatorname{div}_{\bm{\theta}} \bm{f}(t,\bm{\theta})
 \end{eqnarray*}
 using Lemmas~\ref{lem:derivative} and~\ref{lem:product}, and the result follows.
\end{proof}

\noindent This result, along with block decomposition defined in the next section, will later allow us to translate Lemma~\ref{lem:ags_polynomial} into a convex programming problem.

\subsection{Block decomposition and positive definiteness}\label{sec:block_decomp_hybrid}

In this section, we define a \textit{block decomposition} of Gram representations~\eqref{eq:gram_matrix} of functions of the form~\eqref{eq:polynomial_mix}. We will establish that this block decomposition paves the way for enforcing positive definiteness of these functions in convex programs.

\begin{definition}[Block decomposition of Gram representation]\label{def:block}
     The Gram representation $V$ of size $\left\llbracket \bm{n_v}\right\rrbracket$ can be realized as a block matrix $V=\left[\operatorname{Block}_{j,k}(V)\right]_{j=1,\ldots,n_{v_t}^{pol}+1}^{k=1,\ldots,n_{v_t}^{pol}+1}$, where each block has size ${\left\llbracket({n_{v_t}^{rot}},\bm{n_{v_\theta}})\right\rrbracket}$. More formally, \begin{equation}
         \begin{aligned}
             \operatorname{Block}_{j,k}(V)=V_{J,K},\ \ \text{ where }J&=\left\{(j - 1) * \left\llbracket({n_{v_t}^{rot}},\bm{n_{v_\theta}})\right\rrbracket + i\ \colon\  i=1,\ldots,\left\llbracket({n_{v_t}^{rot}},\bm{n_{v_\theta}})\right\rrbracket\right\}\\ 
         \text{ and }K&=\left\{(k - 1) * \left\llbracket({n_{v_t}^{rot}},\bm{n_{v_\theta}})\right\rrbracket + i\ \colon \ i=1,\ldots,\left\llbracket({n_{v_t}^{rot}},\bm{n_{v_\theta}})\right\rrbracket\right\}.
         \end{aligned}
     \end{equation}
     Any block $B$ of size ${\left\llbracket({n_{v_t}^{rot}},\bm{n_{v_\theta}})\right\rrbracket}$ can further be realized as $B=\left[\operatorname{SubBlock}_{p,q}(B)\right]_{j=1,\ldots,n_{v_t}^{rot}+1}^{k=1,\ldots,n_{v_t}^{rot}+1}$, where each sub-block has size $\left\llbracket\bm{n_{v_\theta}}\right\rrbracket$. More formally,
     \begin{equation}
         \begin{aligned}
             \operatorname{SubBlock}_{p,q}(B)=B_{P,Q},\ \ \text{ where }P&=\left\{(p - 1) * \left\llbracket\bm{n_{v_\theta}}\right\rrbracket + r\ \colon\  r=1,\ldots,\left\llbracket\bm{n_{v_\theta}}\right\rrbracket\right\}\\ 
         \text{ and }Q&=\left\{(q - 1) * \left\llbracket \bm{n_{v_\theta}}\right\rrbracket + r\ \colon \ r=1,\ldots,\left\llbracket\bm{n_{v_\theta}}\right\rrbracket\right\}.
         \end{aligned}
     \end{equation}
\end{definition}

\begin{example}\label{examp:block_example}
    The function $v(t,\theta)=(t^2+1)(1-\cos\theta)$ is polynomial in $t$ with $\bm{deg_v}=[2,0,1]$, hence representation size vector $\bm{n_v}=[1,0,1]$. A Gram representation $V$ of $v(t,\theta)$ of size $\left\llbracket\bm{n_v}\right\rrbracket=4$ is given by \begin{eqnarray*}
        (t^2+1)(1-\cos\theta)
        &=&\psi_{1,0,1}(t,\theta)^\dagger\left[
                \begin{array}{cc|cc}
                1 & -1/2 & 0 & 0\\
                -1/2 & 0 & 0 & 0\\ \hline
                0 & 0 & 1 & -1/2\\
                0 & 0 & -1/2 & 1
                \end{array}
            \right]\psi_{1,0,1}(t,\theta)\\
        &=&\psi_{1,0,1}(t,\theta)^\dagger \begin{bmatrix}
            \operatorname{Block}_{1,1}(V) & \operatorname{Block}_{1,2}(V)\\
            \operatorname{Block}_{2,1}(V) & \operatorname{Block}_{2,2}(V)
        \end{bmatrix}\psi_{1,0,1}(t,\theta),
    \end{eqnarray*}
    where each $\operatorname{Block}_{i,j}(V)$ is submatrix of size $\left\llbracket({n_{v_t}^{rot}},\bm{n_{v_\theta}})\right\rrbracket=\left\llbracket (0,1)\right\rrbracket=2$. Moreover, since $n_{v_t}^{rot}=0$, $\operatorname{SubBlock}_{1,1}(\operatorname{Block}_{i,j}(V))=\operatorname{Block}_{i,j}(V)$ for all $1\le i,j\le n_{v_t}^{pol}+1$.

    The function $q(t,\theta)=t\sin t\sin\theta$ is mixed (polynomial-trigonometric) in $t$ with $\bm{deg_q}=[1,1,1]$, hence representation size vector $\bm{n_q}=[1,1,1]$. A maximally-sparse Gram representation of size $\left\llbracket \bm{n_q}\right\rrbracket=8$ is
    \begin{equation*}\setlength{\arraycolsep}{6pt}
        Q=
\left[
\begin{array}{cc:cc|cc:cc}
 \phantom{-}0& \phantom{-}0& \phantom{-}0& \phantom{-}0& \phantom{-}0& \phantom{-}0& \phantom{-}0&-1/4\\
 \phantom{-}0& \phantom{-}0& \phantom{-}0& \phantom{-}0& \phantom{-}0& \phantom{-}0&1/4& \phantom{-}0\\
\hdashline
 \phantom{-}0& \phantom{-}0& \phantom{-}0& \phantom{-}0& \phantom{-}0& \phantom{-}0& \phantom{-}0& \phantom{-}0\\
 \phantom{-}0& \phantom{-}0& \phantom{-}0& \phantom{-}0& \phantom{-}0& \phantom{-}0& \phantom{-}0& \phantom{-}0\\ \hline
 \phantom{-}0& \phantom{-}0& \phantom{-}0& \phantom{-}0& \phantom{-}0& \phantom{-}0& \phantom{-}0& \phantom{-}0\\
 \phantom{-}0& \phantom{-}0& \phantom{-}0& \phantom{-}0& \phantom{-}0& \phantom{-}0& \phantom{-}0& \phantom{-}0\\ \hdashline
 \phantom{-}0&1/4& \phantom{-}0& \phantom{-}0& \phantom{-}0& \phantom{-}0& \phantom{-}0& \phantom{-}0\\
-1/4& \phantom{-}0& \phantom{-}0& \phantom{-}0& \phantom{-}0& \phantom{-}0& \phantom{-}0& \phantom{-}0
\end{array}
\right].
    \end{equation*}
 Herein, blocks of size $\left\llbracket (n_{q_t}^{rot},\bm{n_{q_\theta}})\right\rrbracket=\left\llbracket (1,1)\right\rrbracket=4$ are separated by solid lines, and their sub-blocks of size $\left\llbracket \bm{n_{q_\theta}}\right\rrbracket=\left\llbracket (1)\right\rrbracket=2$ are separated by dashed lines.
\end{example}

Due to the structure of the standard basis~\eqref{eq:Standard_basis_hybrid}, the block decomposition enables us to obtain periodic functions of $\bm\theta$ occurring as function-valued coefficients of $t^c {\rm{e}}^{-idt}$ in $v(t,\bm{\theta})$. The above characterization also yields a sufficient condition for obtaining a positive-definite function $v(t,\bm{\theta})$ via convex programming. It is given in the following proposition.

\begin{proposition}\label{prop:PD_condition_gram}
    Given a function $v(t,\bm{\theta})$ of the form~\eqref{eq:polynomial_mix} and its Gram representation $V$ of size $\left\llbracket\bm{n_v}\right\rrbracket$, the function can be written as \begin{eqnarray*}
        v(t,\bm{\theta})=\sum_{c=0}^{2n_{v_t}^{pol}}\sum_{d=-n_{v_t}^{rot}}^{n_{v_t}^{rot}}t^c{\rm{e}}^{-idt}\, \left(\psi_{\bm{n_{v_\theta}}}(\bm{\theta})^\dagger\left(\sum_{p-q=d} \operatorname{SubBlock}\left(\sum_{j+k=c+2} \operatorname{Block}_{j,k}(V)\right)\right)\psi_{\bm{n_{v_\theta}}}(\bm{\theta})\right).
    \end{eqnarray*}
    
    \noindent Consequently, the function $v(t,\bm{\theta})$ satisfies $v(t,\bm{0})=0$ if and only if \begin{eqnarray}\label{eq:iterated_block_main_prop}
        \bm{1}^\top\left(\sum_{p-q=d} \operatorname{SubBlock}\left(\sum_{j+k=c+2} \operatorname{Block}_{j,k}(V)\right)\right)\bm{1}=0,
    \end{eqnarray}
    for each $c=0,\ldots,2n_{v_t}^{pol}$ and $d=-n_{v_t}^{rot},\ldots,n_{v_t}^{rot}$. In addition, if $v(t,\bm{\theta})\ge \varepsilon\, \gamma(\bm{\theta})$ for some $\varepsilon>0$ and $\gamma(\bm{\theta})=n-\sum_{i=1}^n \cos(\theta_i)$, then $v(t,\bm{\theta})$ is positive definite.
\end{proposition}
    
\begin{proof}
    We temporarily introduce the notations $\psi^{rot}_{\bm{n_{v}}}=\psi_{n_{v_t}^{rot}}(t)\otimes\psi_{\bm{n_{v_\theta}}}(\bm{\theta})$ and $B_c=\sum_{j+k=c+2} \operatorname{Block}_{j,k}(V)$ for this proof. Using~\eqref{eq:gram_matrix} and Definition~\ref{def:block}, we have
    \begin{eqnarray*}
    &&\hspace{-30pt}v(t,\bm{\theta})\\
    &=&\psi_{\bm{n_v}}(t,\bm{\theta})^\dagger\, V\, \psi_{\bm{n_v}}(t,\bm{\theta})\\
    &=& \begin{bmatrix}
       \psi^{rot}_{\bm{n_{v}}}\\ t\,\psi^{rot}_{\bm{n_{v}}}\\
        \vdots\\
        t^{n_{v_t}^{pol}}\,\psi^{rot}_{\bm{n_{v}}}
    \end{bmatrix}^\dagger\,\begin{bmatrix}
        \operatorname{Block}_{1,1}(V) & \operatorname{Block}_{1,2}(V) & \ldots & \operatorname{Block}_{1,n_{v_t}^{pol}+1}(V)\\
        \operatorname{Block}_{2,1}(V) & \operatorname{Block}_{2,2}(V) & \ldots & \operatorname{Block}_{2,n_{v_t}^{pol}+1}(V)\\
        \vdots & \vdots &  \ddots & \vdots\\
        \operatorname{Block}_{n_{v_t}^{pol}+1,1}(V) & \operatorname{Block}_{n_{v_t}^{pol}+1,2}(V) & \ldots & \operatorname{Block}_{n_{v_t}^{pol}+1,n_{v_t}^{pol}+1}(V)\\
    \end{bmatrix}\, \begin{bmatrix}
        \psi^{rot}_{\bm{n_{v}}}\\ t\,\psi^{rot}_{\bm{n_{v}}}\\
        \vdots\\
        t^{n_{v_t}^{pol}}\,\psi^{rot}_{\bm{n_{v}}}
    \end{bmatrix}\\
    &=& \sum_{j,k=1}^{n_{v_t}^{pol}+1} t^{\/j+k-2} \left({\psi^{rot}_{\bm{n_{v}}}}^\dagger\, \operatorname{Block}_{j,k}(V)\,\psi^{rot}_{\bm{n_{v}}}\right)\\
    &=& \sum_{c=0}^{2n_{v_t}^{pol}}t^c\, \left({\psi^{rot}_{\bm{n_{v}}}}^\dagger\left(\sum_{j+k=c+2} \operatorname{Block}_{j,k}(V)\right){\psi^{rot}_{\bm{n_{v}}}}\right)\\
    &=& \sum_{c=0}^{2n_{v_t}^{pol}}t^c\, \left({\psi^{rot}_{\bm{n_{v}}}}^\dagger\, B_c \, {\psi^{rot}_{\bm{n_{v}}}}\right),
\end{eqnarray*}
where each term in brackets can be expanded as
\begin{eqnarray*}
    &&\hspace{-30pt}{\psi^{rot}_{\bm{n_{v}}}}^\dagger\, B_c \,{\psi^{rot}_{\bm{n_{v}}}}\\
    &=& \begin{bmatrix}
        \psi_{\bm{n_{v_\theta}}}(\bm{\theta})\\
        \vdots\\
        {\rm{e}}^{i n_{v_t}^{rot} t}\,\psi_{\bm{n_{v_\theta}}}(\bm{\theta})
    \end{bmatrix}^\dagger\,\begin{bmatrix}
        \operatorname{SubBlock}_{1,1}(B_c) & \ldots & \operatorname{SubBlock}_{1,n_{v_t}^{rot}+1}(B_c)\\
        \vdots &  \ddots & \vdots\\
        \operatorname{SubBlock}_{n_{v_t}^{rot}+1,1}(B_c) & \ldots & \operatorname{SubBlock}_{n_{v_t}^{rot}+1,n_{v_t}^{rot}+1}(B_c)\\
    \end{bmatrix}\, \begin{bmatrix}
        \psi_{\bm{n_{v_\theta}}}(\bm{\theta})\\ 
        \vdots\\
        {\rm{e}}^{i n_{v_t}^{rot} t}\,\psi_{\bm{n_{v_\theta}}}(\bm{\theta})
    \end{bmatrix}\\
    &=& \sum_{p,q=1}^{n_{v_t}^{rot}+1} {\rm{e}}^{i (q-p) t} \left(\psi_{\bm{n_{v_\theta}}}(\bm{\theta})^\dagger\, \operatorname{SubBlock}_{p,q}(B_c)\,\psi_{\bm{n_{v_\theta}}}(\bm{\theta})\right)\\
    &=& \sum_{d=-n_{v_t}^{rot}}^{n_{v_t}^{rot}}{\rm{e}}^{-i d t}\, \left(\psi_{\bm{n_{v_\theta}}}(\bm{\theta})^\dagger\left(\sum_{p-q=d} \operatorname{SubBlock}_{p,q}(B_c)\right)\psi_{\bm{n_{v_\theta}}}(\bm{\theta})\right).
\end{eqnarray*}
We, thus, obtain the first part of the proposition.

Since $\psi_{\bm{n_{v_\theta}}}(\bm{0})=\bm{1}$, $v(t,\bm{0})=0$ for all $t$ if and only if $\bm{1}^\top\left(\sum_{p-q=d} \operatorname{SubBlock}\left(\sum_{j+k=c+2} \operatorname{Block}_{j,k}(V)\right)\right)\bm{1}=0$ for each $c=0,\ldots,2n_{v_t}^{pol}$ and $d=-n_{v_t}^{rot},\ldots,n_{v_t}^{rot}$. Moreover if $v(t,\bm{\theta})\ge \varepsilon\, \gamma(\bm{\theta})$, then $v(t,\bm{\theta})>0$ for all $t$ and $\bm{\theta}\ne \bm{0}$ since $\gamma(\bm{\theta})\ge 0$ and equality holds precisely at $\bm\theta=\bm{0}$. Thus, $v(t,\bm{\theta})$ is positive definite. 
\end{proof}

\begin{remark}\label{rem:zero_sum_hybrid}
    The condition $\bm{1}^\top A\bm{1}$ is equivalent to $\sum_{j,k}A_{j,k}=0$, that is, the matrix $A$ has zero-sum. The Gram representation $V$ in Example~\ref{examp:block_example} satisfies \begin{eqnarray*}
        \sum_{j+k=2} \operatorname{Block}_{j,k}(V)=\operatorname{Block}_{1,1}(V)&=&\begin{bmatrix}
            1 & -1/2\\ -1/2 & 0
        \end{bmatrix},\\
         \sum_{j+k=3} \operatorname{Block}_{j,k}(V)=\operatorname{Block}_{1,2}(V)+\operatorname{Block}_{2,1}(V)&=&\begin{bmatrix}
            0 & 0\\ 0 & 0
        \end{bmatrix},\\
        \sum_{j+k=4} \operatorname{Block}_{j,k}(V)=\operatorname{Block}_{2,2}(V)&=&\begin{bmatrix}
            1 & -1/2\\ -1/2 & 0
        \end{bmatrix}.
    \end{eqnarray*}
    Since all of these matrices have zero-sum, $v(t,\bm{0})=0$ for all $t$ according to Proposition~\ref{prop:PD_condition_gram}. This can also be verified by the expression of $v(t,\bm{\theta})$ given in Example~\ref{examp:block_example}. The reader should note that, when $n_{v_t}^{rot}=0$, the condition~\eqref{eq:iterated_block_main_prop} means that the sum of all block matrices (of size $\left\llbracket(0,1)\right\rrbracket$) on any antidiagonal should be zero-sum. The condition on blocks for a general $\left\llbracket\bm{n_v}\right\rrbracket$ can be visualized using Figure~\ref{fig:block_condition_visual}(i). 

    Similarly, the Gram representation $Q$ in Example~\ref{examp:block_example} satisfies 
    \begin{eqnarray*}
        \sum_{j+k=2} \operatorname{Block}_{j,k}(Q)=\sum_{j+k=4} \operatorname{Block}_{j,k}(Q)&=&
            \left[\begin{array}{cc:cc}
        0 & 0 & 0 & 0\\
        0 & 0 & 0 & 0\\ \hdashline
        0 & 0 & 0 & 0\\
        0 & 0 & 0 & 0\\
        \end{array}\right],\\
        \sum_{j+k=3} \operatorname{Block}_{j,k}(Q)&=& \left[\begin{array}{cc:cc}
        0 & 0 & 0 & -1/4\\
        0 & 0 & 1/4 & 0\\ \hdashline
        0 & 1/4 & 0 & 0\\
        -1/4 & 0 & 0 & 0\\
        \end{array}\right].
    \end{eqnarray*}
    We subsequently have for each $c=0,1,2$ and $d=-1,0,1$
    \begin{eqnarray*}
        \sum_{p-q=d} \operatorname{SubBlock}\left(\sum_{j+k=c+2} \operatorname{Block}_{j,k}(Q)\right)\in\left\{\pm\begin{bmatrix}
            0 & -1/4\\
            1/4 & 0
        \end{bmatrix},\begin{bmatrix}
            0 & 0\\
            0 & 0
        \end{bmatrix}\right\}.
    \end{eqnarray*}
    Since all the matrices in the set have zero-sum, $q(t,\bm{0})=0$ for all $t$ according to Proposition~\ref{prop:PD_condition_gram}.  The condition on blocks for a general representation-size vector can be visualized using Figure~\ref{fig:block_condition_visual}(iii).
\end{remark}

Propositions~\ref{prop:density_condition} and~\ref{prop:PD_condition_gram} allow us to obtain the SDP formulation of the construction of a time-varying Lyapunov density for nonautonomous systems on the hypertorus that are mixed (polynomial-trigonometric) in time. The result is presented in the following section.

\subsection{The SDP formulation}\label{sec:SDP_formulation_thm_hybrid}

In this section, we obtain an SDP to construct the functions $v(t,\bm{\theta})$ and $w(t,\bm{\theta})$ satisfying the conditions of Lemma~\ref{lem:ags_polynomial}, which can be solved using any modern convex programming solvers. The result is given below.

\begin{theorem}[For systems which are not purely trigonometric in $t$]\label{thm:ags}
    Consider the nonautonomous system~\eqref{eq:system_timevar} which is polynomial, or mixed in $t$; and trigonometric in $\bm\theta$, such that all its components $f^{(l)}$ can be represented by matrices of size $\left\llbracket\bm{n_f}\right\rrbracket$ with $n_{f_t}^{pol}\ne 0$. For a given $\varepsilon>0$ and $\bm{n_v}$ satisfying $n_{v_t}^{pol}\ge n_{f_t}^{pol}+1$, if there exist Gram matrices $V\ge 0$ and $W\ge 0$ of sizes $\left\llbracket\bm{n_v}\right\rrbracket$ and $\left\llbracket\bm{n_v}+\bm{n_f}\right\rrbracket$, respectively, satisfying \begin{enumerate}[label=(C\arabic*)]
        \item\label{item:C1} the linear constraint~\eqref{eq:LieRelatonDual_Gram};
        \item\label{item:C2} $\bm{1}^\top\left(\sum_{p-q=d} \operatorname{SubBlock}\left(\sum_{j+k=c+2} \operatorname{Block}_{j,k}(V)\right)\right)\bm{1}=0$
    for each $c=0,\ldots,2n_{v_t}^{pol}$ and $d=-n_{v_t}^{rot},\ldots,n_{v_t}^{rot}$; and 
        \item\label{item:C3} $V\ge \varepsilon\,\Gamma$, where $\Gamma$ is a Gram representation of $\gamma(\bm{\theta})=n-\sum_{i=1}^n \cos(\theta_i)$ of size $\left\llbracket\bm{n_v}\right\rrbracket$,
    \end{enumerate}
    then for every initial time $t_0$, the set of points $\bm{\theta_0}$ that do not satisfy $\lim_{t\to\infty}\bm\Phi(t;t_0,\bm{\theta_0})=\bm{0}$ has zero Lebesgue measure in $\mathbb{T}^n$.
\end{theorem}

\begin{proof}
    Since $V,W\ge 0$, the functions $v(t,\bm{\theta})$ and $w(t,\bm{\theta})$ are positive semidefinite by Lemma~\ref{lem:PD_dumitrescu}. Note that~\textit{\ref{item:C3}} implies $v(t,\bm{\theta})\ge \varepsilon\, \gamma(\bm{\theta})$.  Thus, $v(t,\bm{\theta})$ is positive definite, due to~\textit{\ref{item:C2}} and~\textit{\ref{item:C3}}, by  Proposition~\ref{prop:PD_condition_gram}. Also, equation~\eqref{eq:LieRelationDual} is satisfied, due to~\textit{\ref{item:C1}}, by Proposition~\ref{prop:density_condition}. Since $\mathbb{T}^n$ is compact, the closed subset $\{\bm{\theta}\in\mathbb{T}^n\colon \left\|\bm{\theta}\right\|\ge r \}$ is compact, hence
    \begin{equation*}
        \widehat{I}_r=\int_0^\infty \int_{\left\|\bm{\theta}\right\|\ge r}\frac{1+\left\|\bm{f}(t,\bm{\theta})\right\|}{(1+\|\bm{\theta}\|)\, v(t,\bm{\theta})}\ \mathrm{d}\bm{\theta}\, \mathrm{d}t<\infty,
    \end{equation*}
    since for large $t$, the integrand behaves like $t^{2\left(n_{f_t}^{pol}-n_{v_t}^{pol}\right)}$ where $n_{v_t}^{pol}\ge n_{f_t}^{pol}+1$. The result then follows from Lemma~\ref{lem:ags_polynomial}.
\end{proof}

\begin{figure}[ht]
\centering
\begin{tikzpicture}[scale=0.7]

\def\p{5}      
\def\m{1}      
\def\t{7}      


\begin{scope}[xshift=7.5cm]

\definecolor{coralB}{RGB}{210,140,120}

\foreach \i in {0,...,4} {
  \foreach \j in {0,...,4} {
  \ifnum\numexpr\i+\j\relax=0
      \fill[teal!80] (\i,\j) rectangle ++(1,1);
      \node at (\i+0.5,\j+0.5) {\normalsize $4$};
    \fi
  \ifnum\numexpr\i+\j\relax=1
      \fill[teal!60] (\i,\j) rectangle ++(1,1);
      \node at (\i+0.5,\j+0.5) {\normalsize $3$};
    \fi
    \ifnum\numexpr\i+\j\relax=2
      \fill[teal!40] (\i,\j) rectangle ++(1,1);
      \node at (\i+0.5,\j+0.5) {\normalsize $2$};
    \fi
      \ifnum\numexpr\i+\j\relax=3
      \fill[teal!20] (\i,\j) rectangle ++(1,1);
      \node at (\i+0.5,\j+0.5) {\normalsize $1$};
    \fi
    \ifnum\numexpr\i+\j\relax=4
      \fill[white!20] (\i,\j) rectangle ++(1,1);
      \node at (\i+0.5,\j+0.5) {\normalsize $0$};
    \fi
    \ifnum\numexpr\i+\j\relax=5
      \fill[coralB!20] (\i,\j) rectangle ++(1,1);
      \node at (\i+0.5,\j+0.5) {\normalsize $-1$};
    \fi
    \ifnum\numexpr\i+\j\relax=6
      \fill[coralB!40] (\i,\j) rectangle ++(1,1);
      \node at (\i+0.5,\j+0.5) {\normalsize $-2$};
    \fi
    \ifnum\numexpr\i+\j\relax=7
      \fill[coralB!60] (\i,\j) rectangle ++(1,1);
      \node at (\i+0.5,\j+0.5) {\normalsize $-3$};
    \fi
    \ifnum\numexpr\i+\j\relax=8
      \fill[coralB!80] (\i,\j) rectangle ++(1,1);
      \node at (\i+0.5,\j+0.5) {\normalsize $-4$};
    \fi
  }
}
\draw[step=\m, black, thin] (0,0) grid (\p*\m,\p*\m);

\foreach \i in {0,...,4} {
    \foreach \j in {0,...,4} {
    \ifnum\numexpr\i+\j\relax=2
    \draw[line width=1.5pt]
        (\i*\m,\j*\m) rectangle ++(\m,\m);
    \fi
}
}
\draw[decorate, decoration={brace, amplitude=5pt}]
(0,5.1) -- (1,5.1)
node[midway, above=6pt] {\footnotesize size $\left\llbracket \bm{n_{v_\theta}}\right\rrbracket$};
\node at (2.5,-1) {(ii) $\operatorname{SubBlock}_{p,q}(B)\colon\ p-q=d$};
\node at (2.5,-1.7) {is one of the blocks marked $d$.};
\end{scope}

\begin{scope}[xshift=0cm]

\foreach \i in {0,...,4} {
  \foreach \j in {0,...,4} {
    \ifnum\numexpr\j-\i\relax=4
      \fill[teal!0] (\i,\j) rectangle ++(1,1);
      \node at (\i+0.5,\j+0.5) {\normalsize $0$};
    \fi
    \ifnum\numexpr\j-\i\relax=3
      \fill[teal!10] (\i,\j) rectangle ++(1,1);
      \node at (\i+0.5,\j+0.5) {\normalsize $1$};
    \fi
    \ifnum\numexpr\j-\i\relax=2
      \fill[teal!20] (\i,\j) rectangle ++(1,1);
      \node at (\i+0.5,\j+0.5) {\normalsize $2$};
    \fi
    \ifnum\numexpr\j-\i\relax=1
      \fill[teal!30] (\i,\j) rectangle ++(1,1);
      \node at (\i+0.5,\j+0.5) {\normalsize $3$};
    \fi
    \ifnum\numexpr\j-\i\relax=0
      \fill[teal!40] (\i,\j) rectangle ++(1,1);
      \node at (\i+0.5,\j+0.5) {\normalsize $4$};
    \fi
    \ifnum\numexpr\j-\i\relax=-1
      \fill[teal!50] (\i,\j) rectangle ++(1,1);
      \node at (\i+0.5,\j+0.5) {\normalsize $5$};
    \fi
    \ifnum\numexpr\j-\i\relax=-2
      \fill[teal!60] (\i,\j) rectangle ++(1,1);
      \node at (\i+0.5,\j+0.5) {\normalsize $6$};
    \fi
    \ifnum\numexpr\j-\i\relax=-3
      \fill[teal!70] (\i,\j) rectangle ++(1,1);
      \node at (\i+0.5,\j+0.5) {\normalsize $7$};
    \fi
    \ifnum\numexpr\j-\i\relax=-4
      \fill[teal!80] (\i,\j) rectangle ++(1,1);
      \node at (\i+0.5,\j+0.5) {\normalsize $8$};
    \fi
  }
}

\draw[step=\m, black, thin] (0,0) grid (\p*\m,\p*\m);
\foreach \i in {0,...,4} {
    \foreach \j in {0,...,4} {
    \ifnum\numexpr\j-\i\relax=1
    \draw[line width=1.5pt]
        (\i*\m,\j*\m) rectangle ++(\m,\m);
    \fi
}
}
\draw[decorate, decoration={brace, amplitude=5pt}]
(0,5.1) -- (1,5.1)
node[midway, above=6pt] {\footnotesize size $\left\llbracket (n_{v_t}^{rot},\bm{n_{v_\theta}})\right\rrbracket$};

\node at (2.5,-1) {(i) $\operatorname{Block}_{j,k}(V)\colon\ j+k=c+2$};
\node at (2.5,-1.7) {is one of the blocks marked $c$.};
\end{scope}

\begin{scope}[xshift=15cm]

\foreach \i in {0,...,4} {
  \foreach \j in {0,...,4} {

    \ifnum\numexpr\i-\j\relax=-1
          \draw[line width=3pt,red,densely dotted]
        (\i*\m,\j*\m) rectangle ++(\m,\m);
    \fi

    \draw[step=\m, black, thin] (0,0) grid (\p*\m,\p*\m);

    \draw[step={1/4}, gray!60, line width=0.3pt] (\i,\j) grid ++(1,1);

    \ifnum\numexpr\i-\j\relax=-1
      \foreach \r in {0,...,3} {
        \foreach \q in {0,...,3} {
          \ifnum\numexpr\q+\r\relax=4
            \fill[red!20]
            (\i + \q/4, \j + \r/4)
            rectangle ++(1/4,1/4);
            \draw[line width=1.5pt]
        (\i + \q/4, \j + \r/4) rectangle ++(1/4,1/4);
          \fi
        }
      }
    \fi

  }
}

\draw[decorate, decoration={brace, amplitude=5pt}]
(0,5.1) -- (0.2,5.1)
node[midway, above=6pt] {\footnotesize size $\left\llbracket \bm{n_{v_\theta}}\right\rrbracket$};

\draw[decorate, decoration={brace, amplitude=5pt}]
(3,5.1) -- (4,5.1)
node[midway, above=6pt] {\footnotesize size $\left\llbracket (n_{v_t}^{rot},\bm{n_{v_\theta}})\right\rrbracket$};

\node at (2.5,-1) {(iii) $\operatorname{SubBlock}_{p,q}\left(\operatorname{Block}_{j,k}(V)\right)\colon$};
\node at (2.5,-1.7) { $p-q=-1$, and $j+k=5$};
\node at (2.5,-2.4) { is one of the red blocks.};

\end{scope}

\end{tikzpicture}
\caption{An illustration of iterated block conditions~\textit{\ref{item:C2}}. The sum of the same colored (or the same numbered) blocks should be a zero-sum matrix due to one of the conditions in~\textit{\ref{item:C2}} of Theorem~\ref{thm:ags} (i) for $c=3$ when $\bm{n_v}=(4,0,\bm{n_{v_\theta}})$, (ii) for $d=2$ when $\bm{n_v}=(0,4,\bm{n_{v_\theta}})$, and (iii) for $c=3$ and $d=-1$ when $\bm{n_v}=(4,3,\bm{n_{v_\theta}})$.}\label{fig:block_condition_visual}
\end{figure}

\begin{remark}[Caveat on the size of $\Gamma$]
    For a function $p(\bm{\theta})$ of representation-size vector $\bm{n_p}$, there are Gram representations of sizes $\left\llbracket\overline{\bm{n_p}}\right\rrbracket$ for any $\overline{\bm{n_p}}\ge {\bm{n_p}}$. For instance, a Gram representation of $v(t,\theta)$ in Example~\ref{examp:block_example} of size $\left\llbracket(1,0,2)\right\rrbracket$ is given by
    \begin{eqnarray*}
        (t^2+1)(1-\cos\theta)&=&\psi_{1,0,2}(t,\theta)^\dagger\left[
                \begin{array}{ccc|ccc}
                1 & -1/2 & 0 & 0 & 0 & 0\\
                -1/2 & 0 & 0 & 0 & 0 & 0\\ 
                0 & 0 & 0 & 0 & 0 & 0 \\ \hline
                0 & 0 & 0 & 1 & -1/2 & 0\\
                0 & 0 & 0 & -1/2 & 1 & 0\\
                0 & 0 & 0 & 0 & 0 & 0 
                \end{array}
            \right]\psi_{1,0,2}(t,\theta).
    \end{eqnarray*}
    In particular, for $\gamma(\bm{\theta})$ defined in Theorem~\ref{thm:ags}, for any $\overline{\bm{n_\gamma}}\ge \bm{n_\gamma}=(0,0,\bm{1})$, there is a Gram representation $\Gamma$, of size $\overline{\bm{n_\gamma}}$. For~\textit{\ref{item:C3}} to make sense, the sizes of $V$ and $\Gamma$ should coincide. The condition $\bm{n_v}\ge \bm{n_{\gamma}}$ ensures that such a comparison is possible, since we can always choose a Gram representation $\Gamma$ of size $\left\llbracket\bm{n_v}\right\rrbracket$.
\end{remark}

Theorem~\ref{thm:ags} does not cover the case when $n_{f_t}^{pol}=0$, that is, when the nonautonomous system~\eqref{eq:system_timevar} is trigonometric in $t$. This is so since the proof utilizes Lemma~\ref{lem:ags_polynomial} and the condition $n_{v_t}^{pol}\ge n_{f_t}^{pol}+1$ is essential in satisfying the integrability condition~\textit{\ref{item:Rintegral}}. However, since Lemma~\ref{lem:ags_polynomial} requires the system and $v(t,\bm\theta)$ to have the same behavior in $t$, we would necessarily require $n_{v_t}^{pol}=n_{f_t}^{pol}=0$ to accommodate this case. The result for the same is presented below.

\begin{theorem}[For systems trigonometric in $t$]\label{thm:ags_trig}
    Consider the nonautonomous system~\eqref{eq:system_timevar} such that all its components $f^{(l)}$ can be represented by matrices of size $\left\llbracket\bm{n_f}\right\rrbracket$ with $n_{f_t}^{pol}=0$. For a given $\bm{n_v}$ with $n_{v_t}^{pol}=0$ and $\varepsilon>0$, if there exist Gram matrices $V\ge 0$ and $W\ge 0$ of sizes $\left\llbracket\bm{n_v}\right\rrbracket$ and $\left\llbracket\bm{n_v}+\bm{n_f}\right\rrbracket$, respectively, satisfying \begin{enumerate}[label=(C\arabic*')]
        \item\label{item:C1'} the linear constraint \begin{eqnarray}\label{eq:LieRelatonDual_Gram_trig}
        \ch{0,m_2}{k}{n_w}{W+X^\dagger\left(\sum_{l=1}^n\left(D_{\theta_l}^\dagger V+VD_{\theta_l}\right)\otimes F^{(l)}-\sum_{l=1}^n V\otimes\left(D_{\theta_l}^\dagger F^{(l)}+F^{(l)}D_{\theta_l}\right)\right)X}\\
        \nonumber=\ch{0,m_2}{k}{n_v}{-D_t^\dagger V-V D_t},
    \end{eqnarray}
    for all $-(\bm{n_{v}}+\bm{n_{f}})\le (0,m_2,\bm{k})\le \bm{n_{v}}+\bm{n_{f}}$.
        \item\label{item:C2'} $\bm{1}^\top\left(\sum_{p-q=d} \operatorname{SubBlock}_{p,q}(V)\right)\bm{1}=0$ for each $d=-n_{v_t}^{rot},\ldots,n_{v_t}^{rot}$; and 
        \item\label{item:C3'} $V\ge \varepsilon\,\Gamma$, where $\Gamma$ is a Gram representation of $\gamma(\bm{\theta})=n-\sum_{i=1}^n \cos(\theta_i)$ of size $\left\llbracket\bm{n_v}\right\rrbracket$,
    \end{enumerate}
    then for every initial time $t_0$, the set of points $\bm{\theta_0}$ that do not satisfy $\lim_{t\to\infty}\bm\Phi(t;t_0,\bm{\theta_0})=\bm{0}$ has zero Lebesgue measure in $\mathbb{T}^n$.
\end{theorem}

\begin{proof}
Note that $n_{v_t}^{pol}=0$, thus $\operatorname{Block}_{1,1}(V)=V$. Thus, all the conditions~\textit{\ref{item:C1'}--\ref{item:C3'}} are their counterparts~\textit{\ref{item:C1}--\ref{item:C3}} in Theorem~\ref{thm:ags} for $n_{v_t}^{pol}=0$. Thus, all arguments except the integrability condition follow. Since $v(t,\bm{\theta})$ is positive definite, the integrand in~\textit{\ref{item:Tintegral}} is continuous on the compact subset $\{\bm{\theta}\in\mathbb{T}^n\colon \left\|\bm{\theta}\right\|\ge r \}$ of $\mathbb{T}^n$ for each fixed $t$, and the integrability condition~\textit{\ref{item:Tintegral}} is satisfied. The result follows from Lemma~\ref{lem:ags_polynomial}.    
\end{proof}

\section{Examples}\label{sec:examples}

In this section, we apply the theory developed in this paper to obtain Lyapunov densities for two nonautonomous systems on the hypertorus: one polynomial in $t$ and the other trigonometric in $t$. We use our program \texttt{NautLDT} solver~\cite{Tripathi_Software_Nonautonomous_Lyapunov_2026} based on MATLAB R2025b to check the feasibility of Theorem~\ref{thm:ags} and Theorem~\ref{thm:ags_trig}, respectively, and construct the functions $v(t,\bm{\theta})$ in both cases. All examples considered in this section arise from the phase-difference dynamics of Kuramoto models with time-varying coupling strength~\cite{guo2021overviews,petkoski2012} and general coupling functions~\cite{tripathi2026certification}. This is elaborated in detail in Example~\ref{examp:Kuramoto}.

\begin{example}[System polynomial in $t$]\label{examp:1}
    Consider the system $\dot{\theta}=-0.2\, t^2\, \sin\theta+0.4\, t\, \sin 2\theta=\vcentcolon f(t,\theta)$. The system has a Gram representation $F$ of size $\left\llbracket\bm{n_f}\right\rrbracket=\left\llbracket(1,0,2)\right\rrbracket=6$, which is a Hermitian matrix with \begin{equation*}
        F_{j,k}=\begin{cases}
            0.1\,i & \text{for }(j,k)=(4,5)\\
            -0.2\,i & \text{for }(j,k)=(1,6)\\
            0 & \text{for other }1\le j\le k\le 9
        \end{cases}\ .
    \end{equation*}
    The SDP developed in Theorem~\ref{thm:ags} is solved with $\bm{n_v}=(2,0,2)$ and $\varepsilon=0.1$ using \texttt{NautLDT} solver~\cite{Tripathi_Software_Nonautonomous_Lyapunov_2026}. The program returns a $9\times 9$ Hermitian matrix $V$ as below \begin{eqnarray*}
        V&=&\left[\begin{array}{ccc|ccc|ccc} 1.569 & -0.4336 & -1.1354 & -0.7438 & 0.1232 & 0.6205 & 0.0856 & -0.0012 & -0.0845\\ -0.4336 & 0.8671 & -0.4336 & 0.1233 & -0.2465 & 0.1233 & -0.0011 & 0.0024 & -0.0011\\ -1.1354 & -0.4336 & 1.569 & 0.6205 & 0.1232 & -0.7438 & -0.0845 & -0.0012 & 0.0856\\ \hline -0.7438 & 0.1233 & 0.6205 & 0.4659 & -0.2224 & -0.2435 & -0.0709 & 0.0506 & 0.0202\\ 0.1232 & -0.2465 & 0.1232 & -0.2224 & 0.4449 & -0.2224 & 0.0507 & -0.1013 & 0.0507\\ 0.6205 & 0.1233 & -0.7438 & -0.2435 & -0.2224 & 0.4659 & 0.0202 & 0.0506 & -0.0709\\ \hline 0.0856 & -0.0011 & -0.0845 & -0.0709 & 0.0507 & 0.0202 & 0.0134 & -0.0147 & 0.0014\\ -0.0012 & 0.0024 & -0.0012 & 0.0506 & -0.1013 & 0.0506 & -0.0147 & 0.0293 & -0.0147\\ -0.0845 & -0.0011 & 0.0856 & 0.0202 & 0.0507 & -0.0709 & 0.0014 & -0.0147 & 0.0134 \end{array}\right].
    \end{eqnarray*}
    The expression for $v(t,{\theta})$ is obtained using~\eqref{eq:gram_matrix} as \begin{eqnarray*}
        v(t,{\theta})&=&4.00506 - 3.46816 t + 1.72372 t^2 - 0.486329 t^3 + 0.0561109 t^4\\
        & & + (-1.73434 + 0.986117 t - 0.898896 t^2 + 0.405496 t^3 - 0.0588274 t^4)\cos\theta\\
        & & +(-2.27072 + 2.48204 t - 0.824822 t^2 + 0.0808333 t^3 + 0.0027206 t^4)\cos2\theta.
    \end{eqnarray*}
    Thus, $\rho(t,\theta)=1/v(t,\theta)$ is a Lyapunov density for the nonautonomous system. The corresponding density $\rho$ is illustrated in Figure~\ref{fig:example1&2} (left). For visual reference, the phase portrait of the associated skew-product system given by $\dot{t}=1$, $\dot{\theta}=f(t,\theta)$ is superimposed, providing insight into how the obtained density is distributed relative to the underlying flow.
\end{example}

\begin{figure}[ht]
		\centering
		\includegraphics[width=0.47\linewidth]{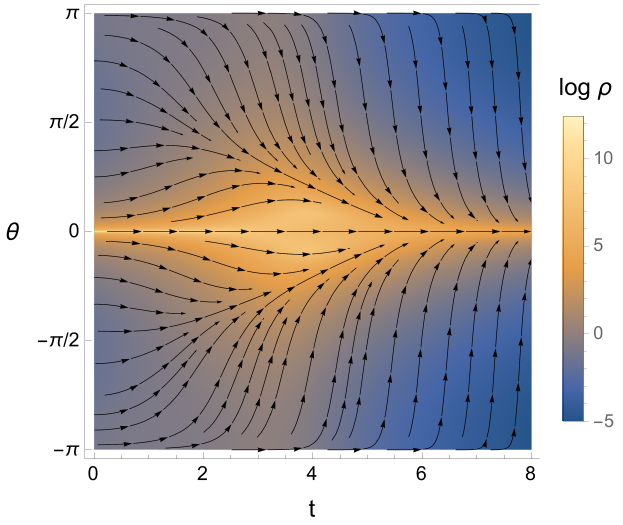}\hspace{20pt}
		\includegraphics[width=0.47\linewidth]{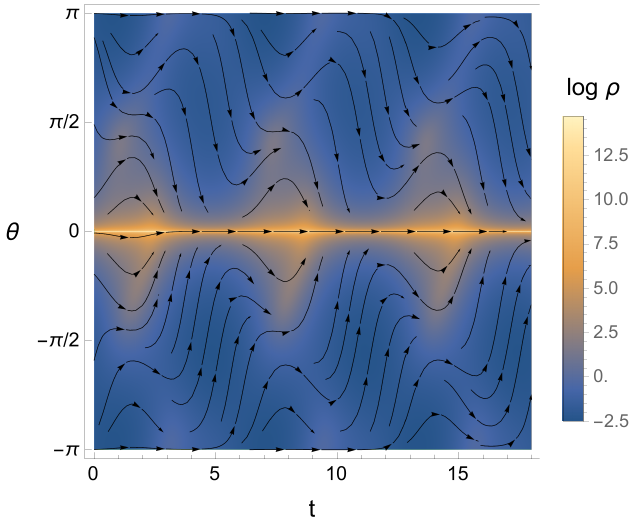}
		\caption{Phase portraits of associated skew-product flow $\dot{t}=1$, $\dot{\theta}=f(t,\theta)$ and logarithmic plot of the density obtained in Example~\ref{examp:1} (left) and Example~\ref{examp:2} with $\bm{\xi}=\bm{1}$ (right).}
		\label{fig:example1&2}
\end{figure}   
 
We now present an example to demonstrate the applicability of Theorem~\ref{thm:ags_trig} to sinusoidal Kuramoto models with periodic coupling, the class studied in~\cite{petkoski2012}.

\begin{example}[Almost-global synchronization of Kuramoto models with periodic coupling]\label{examp:Kuramoto}
    Consider the system \begin{equation}\label{eq:Kuramoto}
        \begin{split}
            \dot{\theta}_1 &=-(1+0.5\sin t)\sin(\theta_1-\theta_2)\\
            \dot{\theta}_2 &=-(1+0.5\sin t)\sin(\theta_2-\theta_1)-(2+0.3\cos t)\sin(\theta_2-\theta_3)\\
            \dot{\theta}_3&=-(2+0.3\cos t)\sin(\theta_3-\theta_2),
        \end{split}
    \end{equation}
    which is a Kuramoto model with interconnection-dependent time-periodic coupling strengths.
    The system is said to exhibit almost-global synchronization if almost all solutions $\bm\Phi(t;t_0,\bm{\theta_0})$ satisfy $\lim_{t\to\infty}\Phi_i(t;t_0,\bm{\theta_0})-\Phi_j(t;t_0,\bm{\theta_0})=0$ for $i,j\in\{1,2,3\}$. Introducing the phase difference variables $\varphi_1=\theta_1-\theta_2$ and $\varphi_2=\theta_2-\theta_3$, we obtain the system \begin{equation}\label{eq:Kuramoto_phdif}
        \begin{split}
            \dot{\varphi}_1&=-2(1+0.5\sin t)\sin\varphi_1-(2+0.3\cos t)\sin\varphi_2\\
            \dot{\varphi}_2&=(1+0.5\sin t)\sin\varphi_1-2(2+0.3\cos t)\sin\varphi_2,
        \end{split}
    \end{equation}
    which is a trigonometric time-varying system. It is known that the almost-global synchronization of~\eqref{eq:Kuramoto} can be established by verifying almost-global stability of the phase-difference system~\eqref{eq:Kuramoto_phdif}. The SDP developed in Theorem~\ref{thm:ags_trig} is solved with $\bm{n_v}=(0,4,[2,2])$ and $\varepsilon=0.1$ using \texttt{NautLDT} solver~\cite{Tripathi_Software_Nonautonomous_Lyapunov_2026}. The program returns a matrix $V$ of size $\left\llbracket(0,4,[2,2])\right\rrbracket=45$, ensuring almost-global stability of~\eqref{eq:Kuramoto_phdif}. The expressions of $V$ and $v(t,\theta)$ are not provided here due to space constraints, but can be computed using \texttt{NautLDT} solver~\cite{Tripathi_Software_Nonautonomous_Lyapunov_2026}. The plot of the trajectories of \eqref{eq:Kuramoto_phdif}, starting at $100$ random initial conditions at time $t=0$, is given in Figure~\ref{fig:example3d}.   \begin{figure}[ht]
		\centering
		\includegraphics[width=0.8\linewidth]{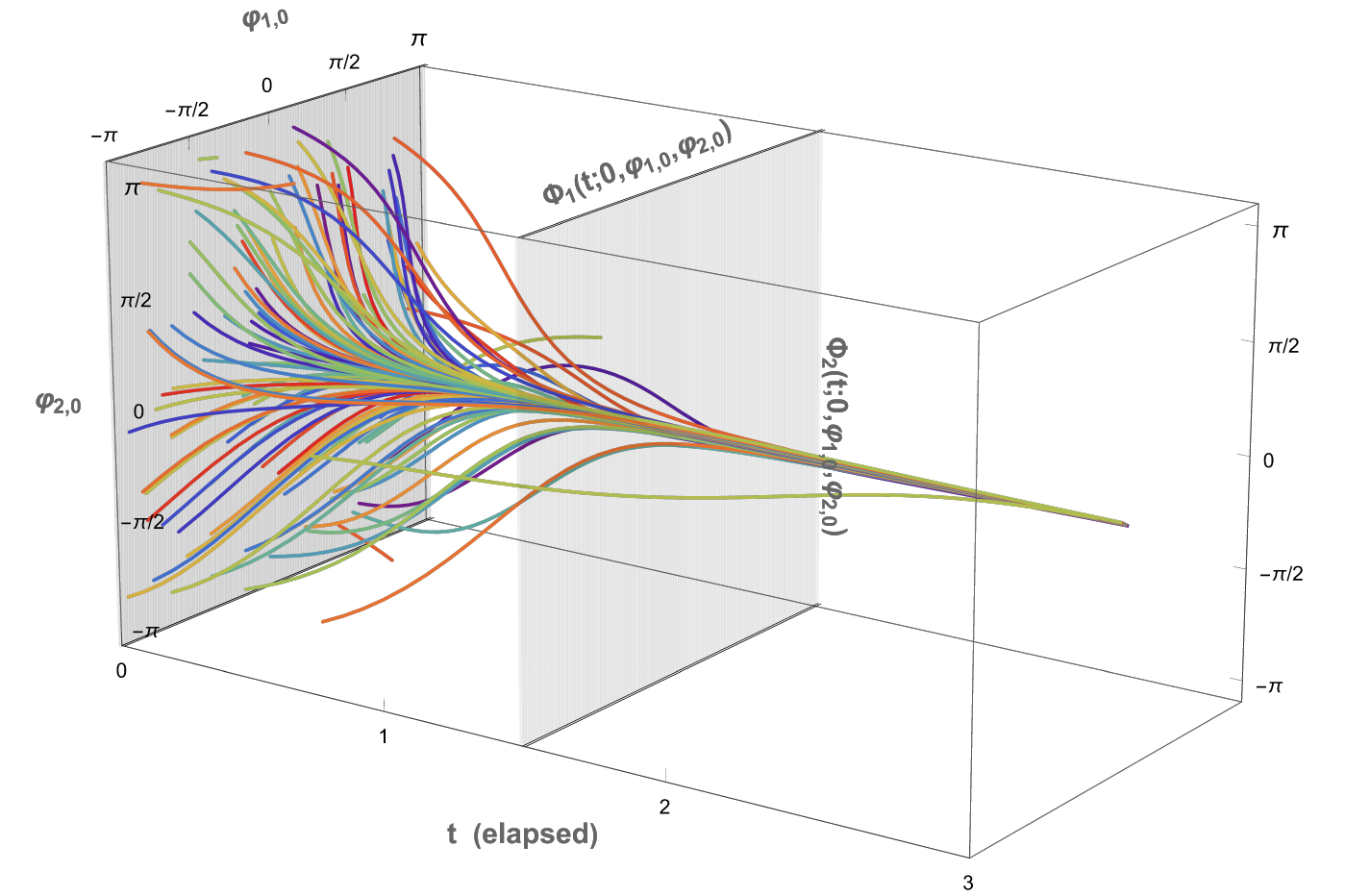}
		\caption{Plot of trajectories of $100$ random initial conditions $(0,\varphi_{1,0},\varphi_{2,0})$ for the phase-difference system~\eqref{eq:Kuramoto_phdif} in Example~\ref{examp:Kuramoto}. All of these trajectories converge to the origin.}
		\label{fig:example3d}
\end{figure}   
\end{example}

\section{Application to robust almost-global stabilization}\label{sec:robust}
We next show how the proposed density-certificate framework can be used to achieve robust almost-global stabilization of nonautonomous systems on the hypertorus. The motivation comes from parameter-varying systems~\cite{briat2014linear}, in which the system depends on uncertain or time-varying quantities, such as noise~\cite{acebron2000uncertainty} and coupling strength~\cite{guo2021overviews,petkoski2012}.

Consider a parameter-varying oscillator system of the form \begin{eqnarray}
    \dot{\bm{\theta}}=\bm{f}(t,\bm{\theta},\bm\xi(t)),
\end{eqnarray}
where $\bm\xi(t)$ denotes an uncertain disturbance signal taking values in a compact set $\Xi$. The objective is to certify convergence of almost all trajectories to the desired equilibrium, uniformly over all admissible parameter trajectories. When the system depends affinely on the uncertain parameter and $\Xi$ is a polytope, the robust verification problem reduces to a finite set of SDP constraints. Indeed, if \begin{eqnarray}
    \dot{\bm\theta}=\bm{f}(t,\bm{\theta},\bm{\xi})=\bm{f_0}(t,\bm\theta)+\sum_{i=1}^p\xi_i\bm{f_i}(t,\bm\theta),
\end{eqnarray}
it is sufficient to obtain a common density $\rho(t,\bm{\theta})$ satisfying the density inequality~\eqref{eq:monzon} simultaneously for systems at the vertices of $\Xi$, similar to the Lyapunov function case~\cite{briat2014linear}. This gives a computationally tractable robust stabilization certificate.

\begin{example}[Robust almost-global stabilization]\label{examp:2}
        Consider the system $\dot{\theta}=-\xi_1\sin\theta+\xi_2\sin t\sin\theta+\cos t\sin 2\theta=\vcentcolon f(t,\theta,\bm{\xi})$, where $\bm{\xi}=(\xi_1,\xi_2)\in[\underline{\xi_1},\overline{\xi_1}]\times[\underline{\xi_2},\overline{\xi_2}]$. The parameter-varying system has a Gram representation of size $\left\llbracket\bm{n_f}\right\rrbracket=\left\llbracket\right(0,1,2)\rrbracket=6$, that is \begin{eqnarray}\label{eq:Gram_rep_PV}
        f(t,\theta,\bm{\xi})&=&\psi_{\bm{n_f}}^{rot}(t,\bm{\theta})^\dagger\,\left[\begin{array}{ccc:ccc}
    0 & i\xi_1/2 & 0 & 0 & -\xi_2/4 & -i/4\\
    -i\xi_1/2 & 0 & 0 & \xi_2/4 & 0 & 0\\ 
    0 & 0 & 0 & i/4 & 0 & 0\\ \hdashline
    0 & \xi_2/4 & -i/4 & 0 & 0 & 0 \\
    -\xi_2/4 & 0 & 0 & 0 & 0 & 0\\
    i/4 & 0 & 0 & 0 & 0 & 0\end{array}\right]\,\psi^{rot}_{\bm{n_f}}(t,\bm{\theta}).
\end{eqnarray}
For $\bm\xi=\bm{1}$, the SDP developed in Theorem~\ref{thm:ags_trig} is solved with $\bm{n_v}=(0,2,3)$ and $\varepsilon=0.1$ using \texttt{NautLDT} solver~\cite{Tripathi_Software_Nonautonomous_Lyapunov_2026}. The program returns a matrix $V$ of size $\left\llbracket(0,2,3)\right\rrbracket=12$. The expressions of $V$ and $v(t,\theta)$ are not provided here due to space constraints, but can be computed using \texttt{NautLDT} solver~\cite{Tripathi_Software_Nonautonomous_Lyapunov_2026}. The plot for the density $\rho=1/v$ obtained is given in Figure~\ref{fig:example1&2} (right).

Fixing this $\bm{n_v}$ and $\varepsilon$, Algorithm~\ref{algo:PVsystems} was run in the \texttt{NautLDT} solver~\cite{Tripathi_Software_Nonautonomous_Lyapunov_2026} to find a common density for systems on the vertex, which are, 
$f(t,\theta,\underline{\xi_1},\underline{\xi_2})$, 
$f(t,\theta,\underline{\xi_1},\overline{\xi_2})$, 
$f(t,\theta,\overline{\xi_1},\underline{\xi_2})$ and
$f(t,\theta,\overline{\xi_1},\overline{\xi_2})$ by varying the values of the bounds $\underline{\xi_j}<1<\overline{\xi_j}$ for $j=1,2$. The largest set (with endpoints up to two decimal places) for which a common density exists was obtained to be $[\underline{\xi_1},\overline{\xi_1}]\times [\underline{\xi_2},\overline{\xi_2}]=[0.8,4.07]\times[0.18,1]$. Thus, almost all solutions of the system $\dot{\theta}=f(t,\theta,\bm{\xi}(t))$ converge to the origin as long as the uncertain parameter $\bm\xi(t)$ stays within this set.

\begin{algorithm}
\caption{Robust almost-global stabilization of $\dot{\bm{\theta}}=\bm{f}(t,\bm{\theta},\xi_1,\ldots,\xi_p)$ via \texttt{NautLDT} solver~\cite{Tripathi_Software_Nonautonomous_Lyapunov_2026}.}\label{algo:PVsystems}
\begin{algorithmic}[1]
\State \textbf{cvx\textunderscore begin}
\State Fix $\underline{\xi_j}$ and $\overline{\xi_j}$ for $j=1,\ldots,p$. Any vertex system has the form $f(t,\bm{\theta},a_1,\ldots,a_p)$ where $a_j=\underline{\xi_j}$ or $\overline{\xi_j}$.
\State Fix Gram representations $F_{(1)},\ldots,F_{(2^p)}$ for vertex systems of size $\left\llbracket\bm{n_f}\right\rrbracket$.
\State Initialize SDP matrix variable $V$ of size $\left\llbracket\bm{n_v}\right\rrbracket$.
\State Initialize SDP matrix variables $W_{(1)},\ldots,W_{(2^p)}$ of sizes $\left\llbracket\bm{n_v}+\bm{n_f}\right\rrbracket$ each.
\State Fix $\varepsilon=0.1$.\Comment{Proposition~\ref{prop:PD_condition_gram}}
\State \hspace{1.5em} {Minimize}$(0)$
\State \hspace{1.5em} subject to
\State \hspace{3em} $V\ge \varepsilon \Gamma$ \Comment{$\Gamma$ as in Theorem~\ref{thm:ags_trig}: to enforce $v\ne 0$ when $\bm\theta\ne\bm{0}$}
\State \hspace{3em} $W_{(1)},\ldots,W_{(2^p)}\ge 0$
\State \hspace{1.5em} \textbf{for} $i=1,\ldots,2^p$
\State \hspace{3em} \textbf{for} {$-(\bm{n_{v}}+\bm{n_{f}})\le (0,m_2,\bm{k})\le \bm{n_{v}}+\bm{n_{f}}$}
    \State \hspace{4.5em}  $\ch{0,m_2}{k}{n_w}{W_{(i)}+X^\dagger\left(\sum_{l=1}^n\left(D_{\theta_l}^\dagger V+VD_{\theta_l}\right)\otimes F_{(i)}^{(l)}-\sum_{l=1}^n V\otimes\left(D_{\theta_l}^\dagger F_{(i)}^{(l)}+F_{(i)}^{(l)}D_{\theta_l}\right)\right)X}$
    \Statex \hspace{25em}$=\ch{0,m_2}{k}{n_v}{-D_t^\dagger V-V D_t}$
\State \hspace{3em} \textbf{end for}\Comment{equation~\eqref{eq:LieRelatonDual_Gram_trig} for $i^{th}$ vertex system.}
\State \hspace{1.5em} \textbf{end for}\Comment{common density for all vertex systems.}
\State \hspace{1.5em} \textbf{for $d=-n_{v_t}^{rot},\ldots,n_{v_t}^{rot}$}
\State \hspace{3em} $\bm{1}^\top\left(\sum_{p-q=d} \operatorname{SubBlock}_{p,q}(V)\right)\bm{1}=0$ \Comment{to enforce $\tilde{v}\equiv 0$ for $\bm\theta=\bm{0}$}
\State \hspace{1.5em} \textbf{end for} 
\State \textbf{cvx\textunderscore end}
\State \textbf{if} cvx\textunderscore status=solved
\State return $V$, $W_{(1)}, \ldots, W_{(2^p)}$.
\end{algorithmic}
\end{algorithm}

To illustrate the result numerically, we generate $100$ random initial conditions $\left\{\big(t_0^{(j)},\theta_0^{(j)}\big),\ j=1,\ldots,100\right\}$, drawn independently and uniformly from the set $[0,10]\times[-\pi,\pi]$. For each experiment $j$, an independent realization of the time-varying parameter vector $\boldsymbol{\xi}^{(j)}(t)=\bigl(\xi_1^{(j)}(t),\xi_2^{(j)}(t)\bigr)$ is generated as follows. At the sampling instants $t_k=2k$, $k=0,1,\ldots,13$, the values $\xi_1^{(j)}(t_k)$ and $\xi_2^{(j)}(t_k)$ are selected independently from the uniform distributions on $[0.8,4.07]$ and $[0.18,1]$, respectively. The parameter functions $\xi_1^{(j)}(t)$ and $\xi_2^{(j)}(t)$ are then defined on the $[0,26]$ by piecewise linear interpolation between consecutive sampling instants.

\begin{figure}[ht]
		\centering
		\includegraphics[width=0.8\linewidth]{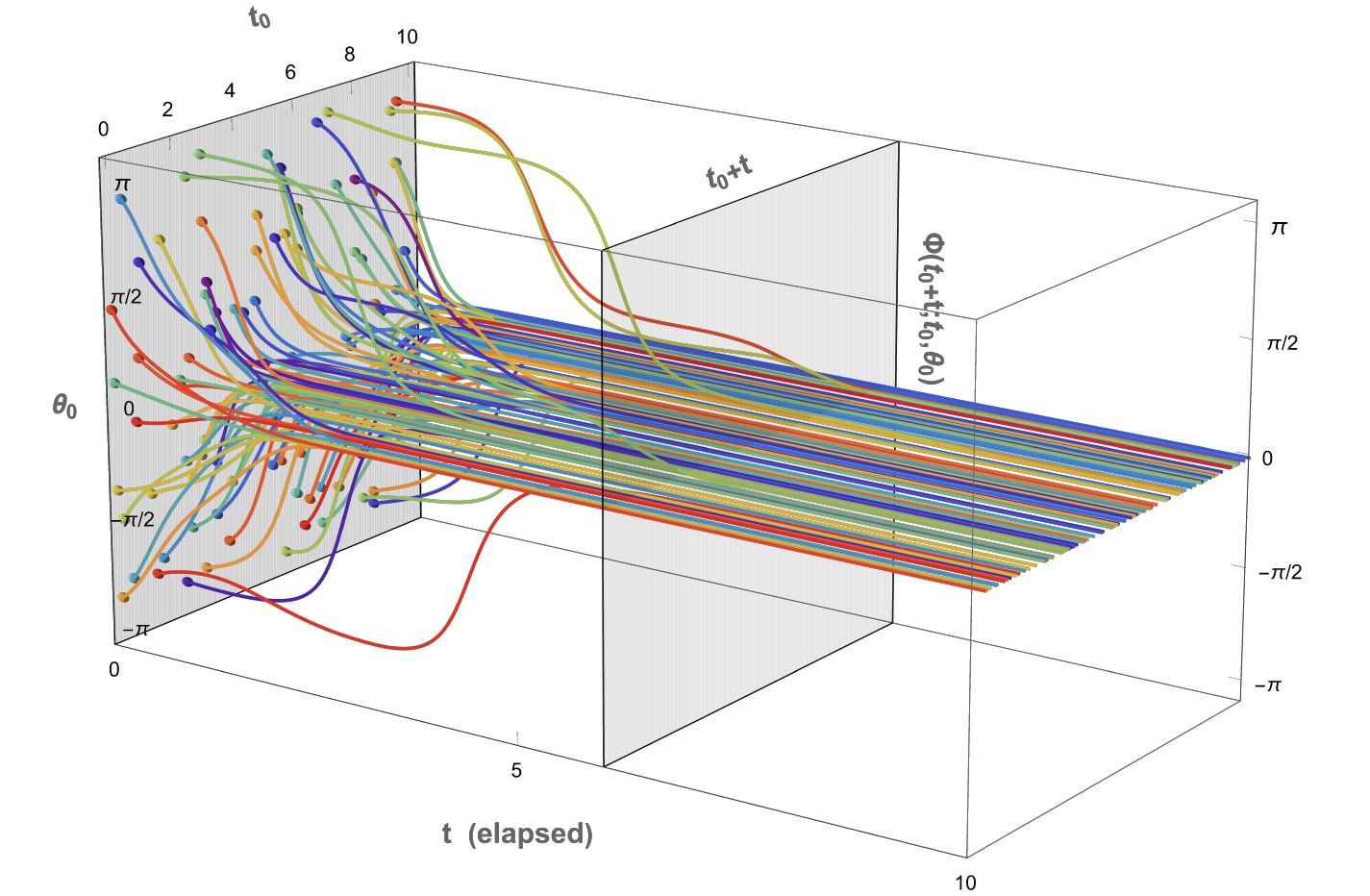}
		\caption{Trajectories of the system $\dot\theta=-\xi_1(t)\sin\theta+\xi_2(t)\sin t\sin\theta+\cos t\sin 2\theta$ for 100 randomly generated initial conditions and time-varying parameters $\xi_1(t)\in[0.8,4.07]$ and $\xi_2(t)\in[0.18,1]$ in Example~\ref{examp:2}. The initial condition plane $(t_0,\theta_0)$ is shown at $t=0$, and the depth of each trajectory represents the evolution of $\Phi(t_0+t;t_0,\theta_0)$ under distinct time-varying parameters.}
		\label{fig:random_over_random_interpolated}
\end{figure}   

The plot in Figure~\ref{fig:random_over_random_interpolated} shows the 100 trajectories of the nonautonomous scalar system $\dot\theta=-\xi_1(t)\sin\theta+\xi_2(t)\sin t\sin\theta+\cos t\sin 2\theta$, where each experiment uses the random time-varying parameters $\xi_1(t)$ and $\xi_2(t)$ obtained above. Each trajectory starts from a randomly chosen initial condition $(t_0,\theta_0)$ obtained above, represented by a colored dot on the initial state-time plane. The horizontal axis shows elapsed time $t$, the vertical axis shows the initial time $t_0$, and the depth axis shows the evolving state $\Phi(t_0+t;t_0,\theta_0)$. Thus, the figure visualizes how solutions evolve under different initial conditions and different realizations of the random time-dependent parameters, revealing the overall asymptotic behavior and robustness of the dynamics with respect to both initial conditions and parameter variations.

\end{example}

\section{Conclusion}

We develop an SDP for constructing Lyapunov densities for nonautonomous systems on the hypertorus with polynomial, trigonometric, or mixed polynomial-trigonometric time dependence. The theory developed is used to obtain a Lyapunov density for a Kuramoto model with interconnection-dependent time-periodic coupling strengths, and is also applied to robust stability of vector fields that are affine in the uncertain parameter. In the latter case, infinitely many density inequalities are reduced to finitely many using vertex vector fields, which is a standard convexity argument, and a computational search is performed to find a common Lyapunov density for all vertex systems simultaneously on the hypertorus. 

The methodology can be extended to nonautonomous systems defined on hypercylinders by leveraging the framework in this paper and~\cite{tripathi_hybrid}; and by proposing an SDP constraint that enforces the integrability conditions. The theory developed here could also be adapted to study linear time-varying systems on Euclidean spaces. 

\section{Availability of Data}

The program supporting the findings of this study is openly available on GitHub~\cite{Tripathi_Software_Nonautonomous_Lyapunov_2026}.

\section*{Conflict of Interest}
The authors have no conflicts of interest to declare that are relevant to the content of this article.

\section*{Acknowledgements}             
This work was supported by the TUBITAK 1001 Research fund [grant number: 122E522].

\bibliographystyle{plain}
\bibliography{mybibfile}
\end{document}